\documentclass[a4paper,12pt]{amsart}
\usepackage{amssymb}
\usepackage[utf8]{inputenc}
\usepackage{amsmath}
\usepackage{stmaryrd}
\usepackage{amscd,amsthm,amssymb}
\usepackage{enumerate}% http://ctan.org/pkg/enumerate
\usepackage{color}
\usepackage[all,cmtip]{xy}

\scrollmode
\usepackage{latexsym}

\def\.{\cdot}

\def\d{{\mathrm d}}

\def\la{\langle}
\def\ra{\rangle}

\def\l{\lambda}

\def\beq{\begin{equation}}
\def\eeq{\end{equation}}
\def\bea{\begin{eqnarray*}}
\def\eea{\end{eqnarray*}}
\def\beaa{\begin{eqnarray}}
\def\eeaa{\end{eqnarray}}
\def\ba{\begin{array}}
\def\ea{\end{array}}
\def\f{\varphi}

\def\o{\omega}

\def \RM{\mathbb{R}}

\def \CM{\mathbb{C}}

\newcommand{\ct}{\nabla^\tau}

\def\Ric{\mathrm{Ric}}
\def\id{\mathrm{id}}
\def\be{\begin{equation}}
\def\ee{\end{equation}}

\def\Hol{\mathrm{Hol}}

\def\Sym{\mathrm{Sym}}
\def\hol{\mathfrak{hol}}
\def\so{\mathfrak{so}}
\def\su{\mathfrak{su}}

\def\sp{\mathfrak{sp}}
\def\gg{\mathfrak{g}}
\def\ll{\mathfrak{l}}
\def\hh{\mathfrak{h}}

\def\k{\mathfrak{k}}
\def\z{\mathfrak{z}}

\def\mm{\mathfrak{m}}

\def\u{\mathfrak{u}}
\def\v{\mathfrak{v}}
\def\SU{\mathrm{SU}}
\def\U{\mathrm{U}}
\def\A{\mathrm{A}}

\def\C{\mathrm{C}}

\def\G{\mathrm{G}}
\def\H{\mathcal{H}}
\def\R{\mathrm{R}}

\def\SO{\mathrm{SO}}

\def\End{\mathrm{End}}

\def\Sp{\mathrm{Sp}}
\def\Spin{\mathrm{Spin}}

\def\Ker{\mathrm{Ker}}
\def\ad{\mathrm{ad}}

\def\Sym{\mathrm{Sym}}

\def\Id{\mathrm{id}}

\def\P{\mathrm{P}}
\def\T{\mathrm{\,T}}

\def\lr{\,\lrcorner\,}

\def\dd{\mathrm{d}}

\def\V{\mathcal{V}}

\def\LieD{{\mathcal L}}
\def\Ad{\mathrm{Ad}\,}

\newtheorem{epr}{Proposition}[section]
\newtheorem{ath}[epr]{Theorem}
\newtheorem{elem}[epr]{Lemma}
\newtheorem{ecor}[epr]{Corollary}

\theoremstyle{definition}
\newtheorem{ede}[epr]{Definition}
\newtheorem{ere}[epr]{Remark}
\newtheorem{exe}[epr]{Example}

\title{Metric connections with parallel skew-symmetric torsion}

\author{Richard Cleyton, Andrei Moroianu, Uwe Semmelmann}

\address{Richard Cleyton, Ribe Katedralskole, Puggårdsgade 22, 6760 Ribe, Denmark}
\email{rc@ribekatedralskole.dk}

\address{Andrei Moroianu, Université Paris-Saclay, CNRS,  Laboratoire de mathématiques d'Orsay, 91405, Orsay, France}
\email{andrei.moroianu@math.cnrs.fr}

\address{Uwe Semmelmann, Institut f\"ur Geometrie und Topologie, Fachbereich Mathematik, Universit{\"a}t Stuttgart, Pfaffenwaldring 57, 70569 Stuttgart, Germany
}
\email{uwe.semmelmann@mathematik.uni-stuttgart.de}

\date{\today}

\begin{document}

\begin{abstract}
A geometry with parallel skew-symmetric torsion is a Riemannian manifold carrying a metric connection with parallel skew-symmetric torsion. Besides the trivial case of the Levi-Civita connection, geometries with non-vanishing parallel skew-symmetric torsion arise naturally in several geometric contexts, e.g. on naturally reductive homogeneous spaces, nearly Kähler or nearly parallel $\mathrm{G}_2$-manifolds, Sasakian and $3$-Sasakian manifolds, or twistor spaces over quaternion-Kähler manifolds with positive scalar curvature.
In this paper we study the local structure of Riemannian manifolds carrying a metric connection with parallel skew-symmetric torsion. On every such manifold one can define a natural splitting of the tangent bundle which gives rise to a Riemannian submersion over a geometry with parallel skew-symmetric torsion of smaller dimension endowed with some extra structure. We show how previously known examples of geometries with parallel skew-symmetric torsion fit into this pattern, and construct several new examples. In the particular case where the above Riemannian submersion has the structure of a principal bundle, we give the complete local classification of the corresponding geometries with parallel skew-symmetric torsion.
\end{abstract}

\subjclass[2010]{53B05, 53C25}
\keywords{Parallel skew-symmetric torsion, quaternion-Kähler structures, Sasakian structures, 3-Sasakian structures, naturally reductive homogeneous spaces}
\maketitle

\section{Introduction}

\subsection{Motivation}

Metric connections with torsion on Riemannian manifolds have been 
studied recently in many articles. Such connections usually  arise in special geometric situations and then 
come with further properties. Besides the (torsion free) Levi-Civita connection, which is, for obvious reasons, the central object in Riemannian geometry, the next most natural class to consider is the  one of metric connections 
with totally skew-symmetric and parallel torsion. These connections have the same 
geodesics as the Levi-Civita connection. Moreover their curvature tensor is still pair symmetric and
satisfies the second Bianchi identity. There are several important geometries admitting
metric connections with parallel skew-symmetric torsion as we will now explain.

\subsection{Examples}
The first example are the naturally reductive spaces with their canonical homogeneous
connection which turns out to have parallel skew-symmetric torsion and also parallel curvature (see \cite{kn2}). 
Another important classes of examples are Sasakian and $3$-Sasakian manifolds \cite{friedrich2} and more generally 3-$(\alpha,\delta)$-Sasakian manifolds \cite{ad}. 
In even dimensions, nearly K\"ahler manifolds with their canonical Hermitian connection
provide another class of examples (see \cite{andrei}). Finally, in dimension 7, every manifold with a nearly parallel $\G_2$-structure carries a canonical connection with parallel skew-symmetric  torsion 
(see \cite{friedrich2}). We will discuss these examples in greater detail in §\ref{examples}.

\subsection{Previous results}
Although not directly related to our topic, let us first mention that the possible holonomy groups of arbitrary torsion-free connections (not necessarily metric) have been classified, under the irreducibility assumption, by Merkulov and Schwachh\"ofer \cite{merkulov}. 

The classification of metric connections with parallel torsion and {\em irreducible} holonomy representation was obtained by Cleyton and Swann in \cite{c-swann}.  They show that a Riemannian manifold admitting a metric 
connection with non-trivial parallel torsion is locally isometric to a non-symmetric isotropy irreducible 
homogeneous space, or to one of the irreducible symmetric spaces $(G\times G)/G$ or $G^\C/G$, 
or it is nearly  K\"ahler in dimension 6, or has a nearly parallel $\G_2$-structure in dimension 7. The homogeneous spaces in the first
case are naturally reductive if the torsion is assumed to be skew-symmetric. For the other three
cases the torsion is automatically skew-symmetric. 

The reducible case turns out to be much more involved, and it is the purpose of the present article to
describe a classification scheme in the case of connections with parallel skew-symmetric torsion whose holonomy representation is reducible.

Further classification results only occur in in special geometric situations or in low dimensions:
Alexandrov, Friedrich and Schoemann \cite{alex1} have shown that if the canonical Hermitian connection of a Hermitian
manifold has parallel torsion and holonomy in $\Sp(n)\U(1)$ then the manifold is locally isomorphic 
to a twistor space over a quaternion-K\"ahler manifold of positive scalar curvature.
Partial classifications of $6$-dimensional almost Hermitian manifolds admitting a canonical Hermitian
connection with parallel torsion are obtained by Alexandrov \cite{alex2} and Schoemann \cite{schoeman}. Similarly,
cocalibrated $\G_2$-manifolds whose characteristic connection has parallel torsion are studied
by Friedrich \cite{friedrich1}.  Moreover, Agricola, Ferreira and Friedrich \cite{agricola} obtained classification results in low dimensions, up  to $6$. 

Finally, we would like to mention the recent work of Storm \cite{storm2}, \cite{storm1} and that of Agricola and Dileo
\cite{ad}. In his thesis, Storm describes a new construction method for naturally reductive spaces  and gives classification results in dimensions 7 and 8. He also has a  general result on metric connections  with skew-symmetric and "reducible" parallel torsion (see Thm. 1.3.5 in \cite{storm1}), similar to 
our Lemma \ref{cr} below. Agricola and Dileo introduce in \cite{ad} a new class of almost $3$-contact metric manifolds
called $3$-$(\alpha, \delta)$-Sasakian manifolds (including as special cases $3$-Sasakian manifolds and quaternionic
Heisenberg groups). They show that these spaces admit a canonical metric connection with skew-symmetric
and parallel torsion (\cite[Thm. 4.4.1]{ad}). 

Several of  the examples mentioned above  are total spaces of Riemannian submersions
over manifolds without torsion, e.g. Sasakian manifolds locally fiber over K\"ahler manifolds, $3$-Sasakian
manifolds locally fiber over quaternion-K\"ahler manifolds and the twistor spaces are $S^2$-fibrations 
over quaternion-K\"ahler manifolds. We will see that this is a general phenomenon which characterizes connections with parallel skew-symmetric torsion with reducible holonomy.

\subsection{Overview}

We now turn to the main results contained in this paper. Let $\nabla^\tau$ be a metric connection with parallel skew-symmetric torsion on a Riemannian manifold $(M,g^M)$. As already mentioned, the case where the holonomy representation of $\nabla^\tau$ is irreducible has been dealt with in \cite{c-swann}, so one can always assume that the holonomy representation is reducible. In contrast to the Riemannian (torsion free) situation, this by no means gives a reduction of the manifold as a Riemannian product (unless the geometry is decomposable, see Definition \ref{ri}), since the de Rham decomposition theorem does not hold for connections with torsion. 

Our first achievement is to show that among all possible parallel distributions of the tangent bundle, there is a particular one which we denote with  $\V M$ and which has the remarkable property that its leaves are totally geodesic, and define locally a Riemannian submersion $M\to N$, which we call the standard submersion. Even more strikingly, the restriction of the curvature of $\tau$ to the leaves of this distribution is $\nabla^\tau$-parallel, so each leaf is a locally homogeneous space by the Ambrose-Singer theorem \cite{ambrose}.

The next step is to show that the fibration $M\to N$ can be obtained as the quotient of a principal bundle over $N$ carrying a connection with {\em parallel curvature} by a subgroup of its structure group. This is achieved as follows. One shows that the holonomy bundle $Q$ of the restriction $\nabla^\tau|_{\V M}$, with structure group $K:=\Hol(\nabla^\tau|_{\V M})$ can be viewed as a principal bundle $P$ over the base $N$ of the standard submersion, with a larger structure group $G$ containing $K$, and such that $G/K$ is a locally homogeneous space isomorphic to the fibers of the standard submersion. This fact can be interpreted as a Ambrose-Singer-like theorem for families, and reduces to the usual Ambrose-Singer theorem when the base $N$ is a point. 

The proof, explained in §\ref{gsts}, is based on the following idea: since $\V M$ is a sub-bundle of the tangent bundle of $M$, the principal bundle $Q$ has a soldering-like form $\theta$, which has the same $K$-equivariance property as the connection form $\alpha$ of $\nabla^\tau|_{\V M}$. The sum $\gamma:=\alpha+\theta$ turns out to define a connection form on $Q$ viewed now as a principal bundle $P\to N$, with larger structure group and smaller basis. In this way, the total spaces of $Q\to M$ and $P\to N$ are the same, but the vertical distribution of $P\to N$ is now the direct sum of  the vertical distribution of $Q\to M$ and of the horizontal lift to $Q$ of the distribution $\V M$. Moreover the curvature of the connection $\gamma$, viewed as a $2$-form on $N$ with values in the adjoint bundle $\ad(P)$ turns out to be parallel and satisfies some further algebraic conditions.
The structure $(P,N,\gamma)$ is called a {\em geometry with parallel curvature} (see Definition \ref{pgwt}). Conversely, this geometry with parallel curvature over $N$ contains enough information in order to recover the whole structure of $M$, as shown in §\ref{inverse}.

In §\ref{gtst} we introduce an important particular case of geometries with parallel curvature, called parallel $\gg$-structures, which correspond to geometries with parallel skew-symmetric torsion whose vertical distribution $\V M$ is spanned by $\nabla^\tau$-parallel vector fields. 

In §\ref{para} we describe several construction methods and examples of parallel $\gg$-structures, and in §\ref{se8} we give the local  classification of manifolds with parallel $\gg$-structures (Theorem \ref{pgs}), which is an important step towards the classification of geometries with parallel skew-symmetric  torsion.

Finally, in the appendix we discuss the 3-$(\alpha,\delta)$-Sasakian structures studied by Agricola and Dileo \cite{ad} in the framework of parallel $\gg$-structures and of geometries with parallel curvature. We explain why the geometries with parallel torsion of special type (i.e. with parallel vertical distribution) only occur when $\delta=2\alpha$, and recover the general case, where $\delta\ne 2\alpha$, including the standard 3-Sasakian situation, by means of geometries with parallel curvature.

 {\sc Acknowledgments.} This work was completed during two ``Research in Pairs'' stays at the MFO, Germany and at the CIRM, France. We warmly thank both institutes for the excellent research conditions offered. We would also like to thank the anonymous referee for several suggestions of improvement and Ilka Agricola for comments on our work and in particular for pointing out recent research related to our results.

\section{Preliminaries}

\subsection{Connections with parallel skew-symmetric torsion}
Let $(M,g)$ be a Riemannian manifold. We will identify as usual vectors and 1-forms or skew-symmetric endomorphisms and 2-forms via the metric $g$. 

In the sequel, if $A$ is a skew-symmetric endomorphism of $\T M$, we will denote by $A\cdot$ the action of $A$ on exterior forms as derivation, given by 
\be\label{adot}A\cdot\omega:=\sum_iAe_i\wedge e_i\lr\omega,\qquad\forall\ \omega\in\Lambda^*\T M\ ,\ee
 in every local orthonormal basis $\{e_i\}$ of $\T M$. Note that if $B$ is another skew-symmetric endomorphism, identified with a 2-form via the metric, then 
$A\cdot B$ is the 2-form corresponding to the commutator of $A$ and $B$, whence:
  \be\label{adotb}A\cdot B=[A,B]=-B\cdot A\ .\ee

Every $3$-form $\tau$ on $M$ can be identified with a tensor of type $(2,1)$ by writing for every $x\in M$
 $$\tau(X,Y,Z) = g(\tau_XY, Z),\qquad\forall\ X,Y,Z\in \T_x M\ .$$ 
In this way, the 2-form $X\lr\tau$ is identified with the skew-symmetric endomorphism $\tau_X$ of $\T_xM$ for every tangent vector $X\in \T_xM$. 

\begin{ede}\label{gwt}
A {\em geometry with parallel skew-symmetric torsion} on $M$ is a Riemannian metric $g$ with Levi-Civita connection $\nabla^g$ and a 3-form $\tau\in\Omega^3(M)$ which is parallel with respect to the metric connection $\nabla^\tau:=\nabla^g+\tau$, i.e. $\nabla^\tau\tau=0$. 
\end{ede}
Of course, since $\tau$ is a $3$-form, $\nabla^\tau$ has skew-symmetric torsion $T^\tau=2\tau$.

Writing $\nabla^g= \nabla^\tau - \tau $ and using the fact that $\tau$ is skew-symmetric and $\nabla^\tau$-parallel, we readily see that the curvature $R^\tau$ satisfies
\begin{equation}\label{rgt}
R^g  = R^\tau  + \tau^2
\qquad \mbox{with}\qquad
(\tau^2)_{X,Y}Z = [\tau_X, \tau_Y]Z - 2 \tau_{\tau_XY}Z \ .
\end{equation}
Taking the scalar product with a vector $W$ in this formula we obtain 

\begin{elem}\label{symm}
Let $\nabla^\tau = \nabla^g + \tau$ be a connection with parallel skew-symmetric torsion $\tau$.
Then
$$
R^g(X,Y,Z,W)=
R^\tau(X,Y,Z,W)  - g(\tau_YZ,\tau_XW) + g(\tau_XZ,\tau_YW) - 2g(\tau_XY, \tau_ZW) \.
$$
In particular the curvature $R^\tau$ is pair symmetric: $R^\tau(X,Y,Z,W)=R^\tau(Z,W,X,Y)$.
\end{elem} 

Taking the cyclic sum over $X,Y,Z$ in the previous relation and using the Bianchi identity for $R^g$ yields at once:
\begin{ecor} For every $X,Y,Z,W\in \T M$, one has 
\be\label{bia}\mathop{\mathfrak{S}}_{XYZ} \left(R^\tau(X,Y,Z,W)- 4g(\tau_XY, \tau_ZW)\right)=0\ .
\ee
\end{ecor}

\subsection{Examples}\label{examples}

As mentioned in the introduction, there are several known families of geometries with parallel skew-symmetric torsion. We review here the most important ones.

\begin{exe}\label{nr}
A homogeneous space $M=G/K$ is called reductive if the Lie algebra $\gg$ of $G$ admits an
$\Ad(K)$-invariant splitting $\gg = \k  \oplus \mm$, where $\mm$ can be identified with the tangent
space to $M$ in the origin $o$. The canonical homogeneous connection 
on the $K$-principal bundle $G\rightarrow G/K$ is defined as the projection onto the Lie algebra $\k$,
i.e. its connection $1$-form $\alpha \in \Omega^1(G,\k)$ is given by $\alpha(X) = X_\k$ for any vector $X\in \gg$.
The connection $\alpha$ induces the canonical homogeneous connection on the tangent bundle of $M$.
Its torsion is given by  $T(X,Y)_o= - [X,Y]_\mm$ for vectors $X,Y\in \mm$. A reductive homogeneous space $M$  
equipped with a $G$-invariant metric $g$ corresponding to a $\Ad(K)$-invariant scalar product
$\la \cdot, \cdot \ra$ on $\mm$ is called {\it naturally reductive} if the torsion of the canonical
homogeneous connection is skew-symmetric, i.e. if
$
\la [X, Y]_\mm, Z\ra + \la Y, [X, Z]_\mm \ra = 0
$
holds for all vectors $X,Y,Z \in \mm$. It is well-known that the canonical homogeneous 
connection has parallel torsion (see \cite{kn2}, Ch. X, Thm. 2.6). 
In this situation not only the torsion but also the curvature is parallel. Conversely, the theorem of Ambrose 
and Singer \cite{ambrose} states that if a metric connection on a Riemannian manifold has parallel skew-symmetric torsion and parallel curvature, then the manifold is locally homogeneous and naturally reductive.
There are many examples of naturally reductive spaces, e.g.  all homogeneous spaces $G/K$ of a 
compact semi-simple Lie group $G$ equipped with the metric induced by the Killing form of $G$.
\end{exe}

\begin{exe}\label{nk} 
Nearly K\"ahler manifolds are almost Hermitian manifolds $(M, g, J)$ where the almost
complex structure $J$ satisfies the condition $(\nabla_X J) X= 0$ for all tangent vectors
$X$. The canonical Hermitian connection $\bar\nabla$ with $\bar \nabla J = 0 = \bar \nabla g$
is defined by $\bar \nabla_XY = \nabla^g_XY - \tfrac12 J(\nabla^g_XJ)Y$. Hence the  nearly K\"ahler 
condition directly implies that the canonical Hermitian connection has skew-symmetric torsion. The torsion is 
also $\bar\nabla$-parallel as was first shown  in \cite{kiri} (see \cite{andrei} for a short proof).
Important examples of nearly K\"ahler manifolds in any dimension $4k+2$ are provided by the twistor spaces of quaternion-K\"ahler
manifolds of positive scalar curvature. Another class of examples are the naturally reductive $3$-symmetric 
spaces (see \cite{gray1}, Prop. 5.6). In \cite{nagy} it is proved that any
strict nearly K\"ahler manifold is locally isometric to a product with factors either $6$-dimensional,  
or homogeneous of a certain type, or a twistor space of a quaternion-K\"ahler manifold of positive
scalar curvature. 
In dimension $6$, Butruille proved that the only homogeneous examples are the naturally reductive 3-symmetric spaces
$S^6, S^3\times S^3, \CM P^3$ and the flag manifold $F(1,2)$ (see \cite{but}). The first non-homogeneous 6-dimensional examples were constructed very recently by Foscolo and Haskins on $S^6$ and $S^3\times S^3$ (see \cite{fh}).
\end{exe}

\begin{exe}\label{sas} 
A {\it Sasakian} structure on a  Riemannian manifold $(M,g)$ is given by a unit length
Killing vector field $\xi$ satisfying the condition 
$
\nabla^g_X \d \xi= - 2 X  \wedge  \xi
$ 
for all tangent vectors $X$. In this situation $\bar\nabla = \nabla^g + \tfrac 12 \xi \wedge \d\xi$
defines a metric connection  with skew-symmetric torsion preserving the Sasakian structure. 
It is easy to show that its torsion $T = \xi \wedge \d\xi$ is $\bar\nabla$-parallel (see \cite{friedrich2}).
Every Sasakian manifold is locally isometric to the total space of an $S^1$-fibre bundle over a
K\"ahler manifold, endowed with the metric given by the Kaluza-Klein construction. 
\end{exe}

\begin{exe}\label{3sas}
A {\it  $3$-Sasakian} manifold is a Riemannian manifold with $3$ unit length Killing vector fields 
satisfying the $\so(3)$-commutator relations, and such that each vector field defines  a 
Sasakian structure. More generally, by rescaling the metric with two positive parameters in the direction of the distribution spanned by the 3 vector fields and of its orthogonal complement, one obtains the so-called 3-$(\alpha,\delta)$-Sasakian manifolds \cite{ad}. Every  $3$-$(\alpha,\delta)$-Sasakian manifold is locally isometric to the
total space of an $\SO(3)$-bundle over a quaternion-K\"ahler manifold of positive
scalar curvature, endowed with a Konishi-type metric \cite{konishi} (see the appendix for details).
\end{exe}

\begin{exe}\label{g2} 
{\it Nearly parallel $\G_2$-manifolds} are Riemannian 7-manifolds carrying a 3-form $\omega$ whose stabilizer at each point is isomorphic to the group $\G_2$ and such that $\d\omega=\l\ast \omega$ for some non-zero real number $\l$. On such manifolds, the
metric connection $\bar\nabla:=\nabla^g+\frac1{12}\l\o$ has parallel 
skew-symmetric torsion \cite{friedrich2}. There are several examples of
homogeneous nearly parallel $\G_2$-manifolds, e.g.  $\SO(5)/\SO(3)$, where the embedding of
$\SO(3)$ into $\SO(5)$ is given by the $5$-dimensional irreducible representation of $\SO(3)$,
or the Aloff-Wallach spaces $\SU(3)/\U(1)_{k,l}$. 
Moreover every $7$-dimensional $3$-Sasakian manifold carries a second Einstein
metric defined by a nearly parallel $\G_2$-structure (see \cite{fiau}).
\end{exe}

\section{The standard decomposition}\label{4}

In contrast to the Riemannian case, there are two different notions of reducibility for geometries with parallel skew-symmetric torsion, which we will explain now. 

\begin{ede}\label{ri} A geometry with parallel skew-symmetric torsion $(M,g,\tau)$ is called:
\begin{itemize}
\item {\em reducible} if the holonomy representation of $\nabla^\tau$ is reducible, and {\em irreducible} otherwise.
\item {\em decomposable} if the tangent bundle of $M$ decomposes in a (non-trivial) orthogonal direct sum of $\nabla^\tau$-parallel distributions $\T M=T_1\oplus T_2$ such that the torsion form satisfies $\tau=\tau_1+\tau_2\in\Lambda^3 T_1\oplus\Lambda^3 T_2$, and {\em indecomposable} otherwise.
\end{itemize}
\end{ede}
 Of course, irreducibility implies indecomposability, but as we will see, there are many examples of indecomposable but reducible geometries with parallel skew-symmetric torsion. 
 
 If $(M_1,g_1,\tau_1)$ and $(M_2,g_2,\tau_2)$ are geometries with parallel skew-symmetric torsion, then their Riemannian product $(M_1\times M_2,g_1+g_2,\tau_1+\tau_2)$ is again a decomposable geometry with parallel skew-symmetric torsion. Conversely, 
 the next result shows that a decomposable geometry with parallel skew-symmetric torsion is always locally a Riemannian product:

\begin{elem}\label{cr} Assume that $(M,g,\tau)$ is decomposable, with $\T M=T_1\oplus T_2$ and $\tau=\tau_1+\tau_2\in\Lambda^3 T_1\oplus\Lambda^3 T_2$. Then $(M,g,\tau)$ is locally isometric to a product of two manifolds with parallel skew-symmetric torsion $(M_i,g_i,\tau_i)$. 
\end{elem}

\begin{proof}
For every vector fields $X\in\Gamma(T_1)$ and $Y\in\Gamma(\T M)$ the assumptions in Definition \ref{ri} yield $\nabla^\tau_YX\in\Gamma(T_1)$ and $\tau(X,Y)\in\Gamma(T_1)$, whence
$$\nabla^g_YX=\nabla^\tau_YX-\tau(Y,X)\in \Gamma(T_1),$$
thus showing that $T_1$ is $\nabla^g$-parallel. Similarly, $T_2$ is $\nabla^g$-parallel. By the local de Rham theorem, $(M,g)$ is locally isometric to a Riemannian product $(M_1,g_1)\times (M_2,g_2)$. 

Moreover, for every $X\in\Gamma(T_1)$ we have
$$\nabla^g_X\tau_2=\nabla^\tau_X\tau_2-\tau_X\cdot\tau_2=-\tau_X\cdot\tau_2=-(\tau_1)_X\cdot\tau_2=0,$$
and similarly $\nabla^g_Z\tau^1=0$ for every vector field $Z\in\Gamma(T_2)$. This shows that $\tau_1$ and $\tau_2$ are projectable on $M_1$ and $M_2$ respectively. It is now clear that $\tau_i$ is parallel with respect to $\nabla^{\tau_i}:=\nabla^{g_i}+\tau_i$, so
$(M_i,g_i)$ are Riemannian manifolds with metric connection $\nabla^{\tau_i}$ with parallel skew-symmetric torsion for $i=1,2$.
\end{proof}

We will now introduce the key notion of the paper, which will eventually lead to the classification results below. This consists in decomposing the holonomy representation into irreducible summands, and attributing to each summand a ``horizontal'' or ``vertical'' character, according to whether there exists or not a non-zero element in the holonomy algebra which acts non-trivially only on that summand. For example, every non-flat irreducible summand in the holonomy representation of the Levi-Civita connection is of horizontal type in this sense, by the de Rham Holonomy Theorem. 

We will see below that the theory of Berger algebras forces the curvature to be parallel along the ``vertical'' summands, which turn out to (locally) define a submersion with totally geodesic and locally homogeneous fibres over a geometry with parallel skew-symmetric torsion of smaller dimension.
To make things precise, let us denote by $n$ the dimension of $M$ and by $\hol\subset\so(n)$ the holonomy algebra of $\nabla^\tau$. 

\begin{ede}\label{33}The representation of $\hol$ on $\RM^n$ decomposes into an orthogonal sum of irreducible $\hol$-modules $\hh_\alpha$ and $\v_j$ such that
each summand $\hh_\alpha$ satisfies $\hol_\alpha:=\so(\hh_\alpha)\cap\hol\not=0$ and each summand $\v_j$ 
satisfies $\so(\v_j)\cap\hol=0$. We define $\hh:=\oplus_\alpha\hh_\alpha$ and $\v:=\oplus_j\v_j.$ \end{ede}

Note that the decomposition in irreducible summands of the holonomy representation of $\hol$ might not be unique, in case there exist isotypical summands. However, it is clear that an isotypical summand is necessarily of vertical type (since any element of $\hol$ would act non-trivially on all other summands isomorphic to it), and thus the subspaces $\hh$ and $\v$ are canonically defined.

\begin{elem}\label{taum} 
For every $\alpha$, the representation of $\hol$ on $\hh_\alpha\otimes\Lambda^2\hh_\alpha^\perp$ has no trivial subspace. 
\end{elem}
\begin{proof} Consider the space 
$$E_\alpha:=\{v\in \hh_\alpha\ |\ Av=0,\ \forall A\in\hol_\alpha\}.$$
Since $\hol_\alpha$ is an ideal of $\hol$, for every $v\in E_\alpha$, $A\in\hol_\alpha$ and $B\in\hol$ we get:
$$A(Bv)=[A,B]v+B(Av)=0,$$
thus showing that $E_\alpha$ is a $\hol$-invariant subspace of $\hh_\alpha$. By the irreducibility of $\hh_\alpha$, we deduce that either $E_\alpha=0$ or $E_\alpha=\hh_\alpha$. The latter case is impossible by the very definition of $\hol_\alpha$. Thus $E_\alpha=0$.

Suppose that $u\in\hh_\alpha\otimes\Lambda^2\hh_\alpha^\perp$ satisfies $B u=0$ for every $B\in\hol$. We write $u=\sum_i v_i\otimes w_i$ where $w_i$ is a basis of $\Lambda^2\hh_\alpha^\perp$. By definition, every $A\in\hol_\alpha$ acts trivially on $\hh_\alpha^\perp$. We thus get $0=Au=\sum_i(A v_i)\otimes w_i$ for every $A\in\hol_\alpha$, whence 
\be\label{av}A v_i =0,\qquad\forall i,\ \forall A\in\hol_\alpha\ .\ee 
Thus $v_i\in E_\alpha=0$ for all $i$, so finally $u=0$.
\end{proof}

Decomposing $\Lambda^3\RM^n$ according to the decomposition $\RM^n=(\oplus_\alpha\hh_\alpha)\oplus(\oplus_j\v_j)= \hh\oplus\v$ and using the above result, we readily obtain that every $\hol$-invariant element of $\Lambda^3\RM^n$ is contained in the following sub-space:
\begin{equation}\label{dectau}(\Lambda^3\RM^n)^\hol\subset(\oplus_\alpha\Lambda^3\hh_\alpha)\oplus(\oplus_\alpha\Lambda^2\hh_\alpha\otimes\v)\oplus\Lambda^3\v\subset \Lambda^3\hh\oplus(\Lambda^2\hh\otimes\v)\oplus\Lambda^3\v.\end{equation}

Corresponding to the decomposition $\RM^n=(\oplus_\alpha\hh_\alpha)\oplus(\oplus_j\v_j)= \hh\oplus\v$ we obtain an orthogonal $\nabla^\tau$-parallel decomposition of the tangent bundle of $M$ as
\begin{equation}
  \label{eq:7}
  \T M=(\oplus_\alpha\H_\alpha M)\oplus(\oplus_j\V_j M)=\H M \oplus \V M \ .\end{equation}
  
\begin{ede}\label{stdec} The above defined decomposition $ \T M=\H M \oplus \V M$ will be referred to as the  {\em standard decomposition} of the tangent bundle of a manifold with parallel skew-symmetric torsion.
\end{ede}

We now decompose the torsion as a sum of $\nabla^\tau$-parallel tensors
\begin{equation*}
  \tau=\tau^\hh+\tau^m+\tau^\v\ ,
\end{equation*}
where $\tau^\hh\in\Lambda^3\H M$, $\tau^\v\in\Lambda^3\V M$ and $\tau^m\in (\Lambda^2\H M\otimes \V M)\oplus (\Lambda^2\V M\otimes \H M)$. 
By \eqref{dectau} we obtain directly:

\begin{elem} \label{tors} The projection of $\tau$ onto $\Lambda^2\V M\otimes \H M $ vanishes, i.e. $\tau^m\in \Lambda^2\H M\otimes \V M$. Moreover, $\tau^\hh$ and $\tau^m$ have further decompositions
$$\tau^\hh=\sum_\alpha\tau^{\hh_\alpha}\in\bigoplus_\alpha\Lambda^3\H_\alpha M,\qquad \tau^m\in\left(\bigoplus_\alpha\Lambda^2\H_\alpha M\right)\otimes\V M.$$
\end{elem}

This has an immediate consequence which we will discuss now.
We will assume for the rest of this section that $(M,g,\tau)$  is a geometry with parallel skew-symmetric torsion and standard decomposition $ \T M=\H M \oplus \V M$.

\begin{elem}\label{tot-geod}
The distribution $\V M$ is the vertical distribution of a locally defined Riemannian submersion $(M,g)\overset{\pi}\to (N,g^N)$
  with totally geodesic fibers.
\end{elem}

\begin{proof}
Lemma \ref{tors} shows that for $V,W$ in $\V M$ we have $\nabla^g_VW=\ct_VW-\tau_VW\in \V M$.  Thus $\V M$ is a totally geodesic integrable distribution. We need to show that the restriction $g^\H$ of the metric $g$ to $\H M$ is  constant along the leaves 
  of $\V M$. The Lie derivative
$$(\LieD_Vg^\H)(A,B) = V(g^\H(A,B)) - g^\H([V,A],B)  - g^\H(A,[V,B])$$
clearly vanishes if $A$ or $B$ is a section of $\V M$. For $X,Y$ sections of $\H M$ we have:
\bea
(\LieD_Vg^\H)(X,Y) &=& V(g^\H(X,Y)) -g^\H([V,X],Y) -g^\H(X,[V,Y])\\
&=& g(\nabla^g_XV,Y)+g(X,\nabla^g_YV) = g(\ct_XV,Y) + g(X,\ct_YV)=0\ .
\eea
This shows that if $N$ denotes the space of leaves of $\V M$ in some neighbourhood of $M$, $g^\H$ projects to a Riemannian metric $g^N$ on $N$.
\end{proof}

\begin{ede}\label{stsub}
The Riemannian submersion $\pi:(M,g) \to (N, g^N)$ defined in Lemma \ref{tot-geod} will be called the {\it standard 
submersion} of a manifold with parallel skew-symmetric torsion.
\end{ede}

The next result is the crucial step for showing that the horizontal part of the torsion is a projectable tensor with respect to the standard submersion:
\begin{elem}\label{tauvtau}
For every vector $V$ in $\V M$ one has $\tau_V\cdot \tau^\hh=0$.
\end{elem}
\begin{proof}
By trilinearity, it is enough to show that $(\tau_V\cdot \tau^\hh)(X,Y,Z)=0$ when each of the vectors $X,Y,Z$ is either in $\V M$ or in $\H M$.
Since $\tau_V$ preserves the splitting $ \T M=\H M \oplus \V M$, the expression 
$$(\tau_V\cdot \tau^\hh)(X,Y,Z)=-\mathop{\mathfrak{S}}_{XYZ}\; \tau^\hh(X,Y, \tau_VZ)$$
vanishes when at least one of $X,Y,Z$ is tangent to $\V M$. Finally, when $X,Y,Z\in\H M$, the Bianchi identity \eqref{bia} together with the fact that $\mathop{\mathfrak{S}}_{XYZ}\; R^\tau(X,Y,Z,V)=0$ (as $\H M$ and $\V M$ are orthogonal $\nabla^\tau$-parallel distributions), yields
$$0=\mathop{\mathfrak{S}}_{XYZ}\; g(\tau_XY, \tau_ZV)=-\mathop{\mathfrak{S}}_{XYZ}\; \tau(X,Y, \tau_VZ)=-\mathop{\mathfrak{S}}_{XYZ}\; \tau^\hh(X,Y, \tau_VZ)=(\tau_V\cdot \tau^\hh)(X,Y,Z)\ .$$
\end{proof}

Using Lemma \ref{tauvtau} we can now prove the announced result:

\begin{elem}\label{projectable}
The horizontal part $\tau^\hh$ of the torsion $\tau$ is projectable to the base $N$ of the standard submersion.
\end{elem}
\proof
It suffices to show that $\LieD_V \tau^\hh = 0$ for all vector fields $V$ in $\V M$.
The torsion $\tau$ is $\nabla^\tau$-parallel, and so are its components $\tau^\hh$, $\tau^m$ and $\tau^\v$. In particular, for every $X\in \T M$ we have
$0= \nabla^\tau_X\tau^\hh = \nabla^g_X\tau^\hh + \tau_X \cdot \tau^\hh$. Moreover, for $V\in\V M$ we have $V\lr \tau^\hh=0$, so we can compute in a local orthonormal basis $\{e_i\}$ of $\T M$ using Lemma \ref{tauvtau} and the Cartan formula:
\bea
\LieD_V \tau^\hh &=&V \lr \dd \tau^\hh=\sum_i V \lr ( e_i \wedge \nabla^g_{e_i}  \tau^\hh)=-\sum_i V \lr ( e_i \wedge \tau_{e_i}\cdot \tau^\hh)\\
&=&
-\tau_V\cdot\tau^\hh+\sum_i e_i\wedge(V\lr(\tau_{e_i}\cdot\tau^\hh))=\sum_i e_i\wedge(\tau_{e_i}\cdot(V\lr\tau^h)-\tau_{e_i}V\lr\tau^\hh)\\
&=&\sum_i e_i\wedge(\tau_V{e_i}\lr\tau^\hh)=-\sum_i \tau_Ve_i\wedge({e_i}\lr\tau^\hh)=-\tau_V\cdot\tau^\hh=0\ .
\eea
\qed

\begin{ere}\label{r411}
By Lemma \ref{projectable} there is a $3$-form $\sigma$ on $N$ such that $\tau^\hh  = \pi^\ast \sigma$. If $\nabla^{g^N}$ denotes the Levi-Civita connection of the metric $g^N$, 
Lemma \ref{tot-geod} shows that $\nabla^\sigma:=\nabla^{g^N}+\sigma$ is a connection with parallel skew-symmetric torsion $T^\sigma=2\sigma$ on $N$. Indeed, the fact that $\nabla^\sigma\sigma=0$ follows immediately from the O'Neill formulas for Riemannian submersions, together with the fact that for every vector $X$ in $\H M$ we have 
$$0=\nabla_X^\tau \tau^\hh=\nabla^g_X\tau^\hh+\tau_X\cdot\tau^\hh=\nabla^g_X\tau^\hh+\tau^\hh_X\cdot\tau^\hh.$$
\end{ere}

For later  use we state the following property of the curvature $R^\tau$, which in particular implies that the component in $\Lambda^2 \H M\otimes \Lambda^2 \V M$ 
of $R^\tau$ is parallel. 

\begin{elem}For every $X,Y\in \H M$ and $V,W\in\V M$, one has 
\be\label{rt}
R^\tau(X,Y,W,V)-4( g(\tau_XY,\tau_WV)  +g(\tau_YW,\tau_XV)  + g(\tau_WX,\tau_YV) )=0\ .
\ee
\end{elem}
\begin{proof}
Since $R^\tau(Y,W,X,V)=R^\tau(W,X,Y,V)=0$ (as $\H M$ and $\V M$ are orthogonal $\nabla^\tau$-parallel distributions), the Bianchi formula \eqref{bia} applied to $Z:=W$ yields:
\bea0&=&\mathop{\mathfrak{S}}_{XYW} \left(R^\tau(X,Y,W,V) -4g(\tau_XY, \tau_WV)\right)\\
&=&R^\tau(X,Y,W,V) - 4\mathop{\mathfrak{S}}_{XYW}(g(\tau_XY, \tau_WV))\\
&=&R^\tau(X,Y,W,V)-4( g(\tau_XY,\tau_WV) +g(\tau_YW,\tau_XV)  + g(\tau_WX,\tau_YV) )\ .
\eea
\end{proof}

We will now make another crucial observation, which will give more information about the fibers of the standard submersion.
Since by Lemma \ref{tot-geod} every fibre $F$ of the standard submersion is totally geodesic, the Levi-Civita connection of $M$ restricts
to the Levi-Civita connection of $F$, and the
connection $\ct$ restricts to a connection $\nabla^F$ with
parallel, skew-symmetric torsion $T^F:=2\tau^\v\vert_F$.  

\begin{epr}\label{rparallel1} The composition of the curvature tensor $R^\tau:\Lambda^2\T M\to \Lambda^2\V M\oplus\Lambda^2\H M$ with the projection on the first summand $\Lambda^2\V M$, is $\nabla^\tau$-parallel. In particular, $\nabla^F$ has parallel curvature and parallel skew-symmetric torsion, so $F$ is (locally) a naturally reductive homogeneous space.
\end{epr}
\begin{proof}
In order to keep the notation simple, we will identify here $\hol$-representations with the associated bundles over $M$, and notice that every $\hol$-invariant object corresponds to a $\nabla^\tau$-parallel section on $M$.
Since $\so(\v_j)\cap\hol$ is trivial for each
irreducible summand $\v_j$ of $\v$, \cite[Thm. 4.4]{c-swann} shows by immediate induction on the number of components $\v_j$ that the space $\mathcal{K}(\hol)$ of algebraic curvature tensors with values in $\hol$ satisfies
\begin{equation}\label{inc}
  \mathcal{K}(\hol)\subset \Sym^2(\hol)\cap\Sym^2(\so(\hh))\subset
  \Sym^2(\so(\hh)).
\end{equation}
Decompose $\Sym^2(\hol)$ orthogonally into $ \mathcal{K}(\hol)\oplus \mathcal{K}(\hol)^\perp$ and write $R^\tau=R^\tau_0+R^\tau_1$ for the
corresponding decomposition of the curvature tensor of $\ct$.  The
Bianchi map
$\mathfrak{b} :\Sym^2(\hol)\to\Lambda^4(\hh\oplus\v)$ is of course $\hol$-invariant. Since its kernel is $\mathcal{K}(\hol)$, there exists a $\hol$-invariant isomorphism called $\mathfrak{b}^{-1}$ from $\mathfrak{b}(\Sym^2(\hol))$ to $\mathcal{K}(\hol)^\perp$, and 
\eqref{rgt} shows that
$R^\tau_1=- \mathfrak{b}^{-1} \mathfrak{b}(\tau^2)$. Consequently, $\ct R^\tau_1=0$ 
since $\tau$  is $\nabla^\tau$-parallel.  

By \eqref{inc} we 
have $\pi_{\so(\v)}\circ R^\tau=\pi_{\so(\v)}\circ R^\tau_1$, where  $\pi_{\so(\v)}\colon \so(\hh\oplus\v)\to \so(\v)$ is the standard projection.  In particular, this shows that
\be\label{par}\ct (\pi_{\so(\v)}\circ R^\tau)=0.\ee

Restricting this equation to $F$, shows that $\nabla^F$
is a metric connection for which both the (skew-symmetric) torsion $T^F$ and
curvature $R^F$ are parallel, so by the Ambrose-Singer theorem \cite{ambrose} $F$ is (locally) a naturally reductive homogeneous space.
\end{proof}

\begin{ere}\label{stdecrem} If the summand $\V M$ in the standard decomposition is trivial (i.e. $\V M=0$), Lemma \ref{cr} and Lemma \ref{tors} show that $(M,g)$ is locally a product of irreducible geometries with torsion. By \cite{c-swann}, each factor is either naturally reductive homogeneous, or has a nearly K\"ahler structure in dimension 6, or a nearly parallel $\G_2$-structure in dimension 7. If $\H M=0$, then $(M,g)$ is locally a naturally reductive homogeneous space by Proposition \ref{rparallel1}. However, the converse is not always true. One can construct naturally reductive homogeneous spaces (e.g. Berger spheres) whose summands $\H M$ and $\V M$ are both non-vanishing (see \cite{storm2}, \cite{storm1} for more details). Note also that for a given reducible geometry with parallel skew-symmetric torsion  (like Sasakian, 3-Sasakian, naturally reductive homogeneous), it is not an easy task to determine explicitly the standard decomposition, since the holonomy of the connection with parallel skew-symmetric torsion is not known in general.
\end{ere}

By the previous remark, since we are interested to study geometries with parallel skew-symmetric torsion other than the homogeneous connection on naturally reductive homogeneous spaces, we will assume from now on that the standard decomposition $\T M=\H M\oplus \V M$ is non-trivial.

\section{Geometries with parallel curvature}\label{5}

We have seen that the base space $N$ of the standard submersion (Definition \ref{stsub})
$M\to N$ of a manifold $M$ with parallel skew-symmetric torsion carries again a geometry with parallel skew-symmetric
torsion. In this section we will show that this geometry carries additional structure.

\subsection{Connections on principal bundles}\label{s41}
For the convenience of the reader we collect here some notation and well known formulas about principal bundles and connections.

Let $K$ be a Lie group, $M$ a manifold of dimension $n$ and $\pi: Q \rightarrow M$ a $K$-principal fibre bundle over $M$. We consider a connection $1$-form
$\alpha \in \Omega^1(Q, \k)$ with curvature form $\Omega^\alpha \in \Omega^2(Q, \k)$, defined as the horizontal part of $\d \alpha$. 
The curvature form is given by 
\beq\label{str1}
\Omega^\alpha  = \d \alpha + \tfrac12 \alpha \wedge \alpha \ ,
\eeq
where $(\alpha\wedge \alpha)(U_1,U_2) := 2[\alpha(U_1),\alpha(U_2)]$ for $U_1,U_2 \in \T Q$.

We define the horizontal and vertical distributions on $Q$ by $\T^{hor}_uQ:=\ker(\alpha_u)$ and $\T^\k_uQ:=\ker(\pi_*)$, so that $\T Q=\T^{hor}Q\oplus \T^\k Q$. A vector field $X$ on $M$ induces a unique vector field $\tilde X$ on $Q$ tangent to $\T^{hor}Q$ such that $X$ and $\tilde X$ are $\pi$-related. This vector field is called the horizontal lift of $X$.

If $V$ is a representation space of $K$ and $VM$ is the associated vector bundle, every element $u\in Q$ defines tautologically a linear isomorphism between $V$ and $VM_{\pi(u)}$ and every
section $\sigma$ of $VM$ determines a $V$-valued function $\hat \sigma$ on $Q$ defined by 
\be\label{hut}\hat \sigma(u):=u^{-1}\sigma_{\pi(u)}.\ee 
The covariant derivative on $VM$ induced by $\alpha$ satisfies
\be\label{nablatauv} (\nabla^\alpha_X\sigma)_{\pi(u)} = u(\tilde X(\hat \sigma)), \qquad\forall u\in Q,\ \forall X\in \T_{\pi(u)}M\ ,\ee
and the curvature tensor of $\nabla^\alpha$ defined by 
$$R^\alpha_{X,Y}\sigma:=\nabla^\alpha_X\nabla^\alpha_Y\sigma-\nabla^\alpha_Y\nabla^\alpha_X\sigma-\nabla^\alpha_{[X,Y]}\sigma$$
for every vector fields $X,Y$ on $M$, is related to the curvature form $\Omega^\alpha$ by the classical formula
\be\label{ral}
\Omega^\alpha_u(\tilde X, \tilde Y)= u^{-1}\circ R^\alpha_{X, Y}\circ u,\qquad\forall u\in Q,\ \forall X,Y\in \T_{\pi(u)}M \ .
\ee

\subsection{The geometry of the standard submersion}\label{gsts}

Let us now return to a geometry with parallel skew-symmetric torsion $(M,g^M,\tau)$, endowed with the connection $\nabla^\tau=\nabla^{g^M}+\tau$.
We denote with $\Hol\subset \SO(n)$ the holonomy group of $\nabla^\tau$ at some point of the orthonormal frame bundle of $M$, and with $\hol$ its Lie algebra.
We consider like before the standard decomposition $\T M = \H M \oplus \V M$ of the tangent bundle of $M$ (which is a $\nabla^\tau$-parallel and orthogonal splitting) and  denote correspondingly by $\RM^n = \hh  \oplus \v$ the $\hol$-invariant decomposition of $\RM^n$. 
From Remark \ref{stdecrem} we can assume that both $\v$ and $\hh$ are non-zero. For every $X\in\T M$ we denote by $X^\H$ and $X^\V$ its orthogonal projections to $\H M$ and $\V M$. Note that the restriction to $\V M$ of the metric $g^M$ induces a scalar product on $\v$.

The connection $\nabla^\tau$ induces a metric connection on the Euclidean vector bundle $\V M$.
We denote by $K\subset\SO(\v)$ the holonomy group of this connection at some orthonormal frame $u:\v\to \V M_x$ of $\V M$, by $\k$ its Lie algebra, by $\pi_M:Q\to M$ the holonomy bundle through $u$, and by $\alpha\in\Omega^1(Q,\k)$ its connection 1-form. Correspondingly, the restriction of $\nabla^\tau$ to $\V M$ will be denoted from now on by $\nabla^\alpha$.

Any Lie algebra element  $A \in \k$ induces a vertical vector field $A^*$ on $Q$ defined
at $u\in Q$ by $ A^*_u := \left.\frac{d}{dt}\right|_{t=0}(u \cdot \exp(tA))$. By definition of the connection we
have $\alpha(A^*)=A$. The vector fields $A^*$ are called {\it fundamental vertical vector fields}. 
It is easy to check that for $A,B \in \k$ we have 
\be\label{br}[A^*,B^*]=[A,B]^*\ .\ee

Since $Q$ is the frame bundle of a sub-bundle of $\T M$, it carries a $1$-form $\theta \in \Omega^1(Q, \v)$ (reminiscent of the soldering form) defined by  
$\theta_u(U) := u^{-1}({\pi_M}_* U)^{\V}$ for every $u\in Q$, (viewed as linear isomorphism
$u : \v \rightarrow \V M_{\pi_M(u)}$) and for every $U\in \T_u Q$. Moreover, every vector $\xi\in\v$ induces a vector field $\xi^*$ on $Q$ defined at $u \in Q$ by $\xi^*_u := \widetilde{u\xi}$ (the horizontal lift of $u\xi\in\V M$ at $u$).
Note that $\theta(\xi^*) = \xi$. The vector field $\xi^*$ is called {\it standard horizontal vector field}. We claim that for $A \in \k$ and $\xi \in \v$ we have 
\be\label{br1}[A^\ast, \xi^\ast] = (A.\xi)^\ast\ ,\ee
where the dot denotes the action of $\k\subset\so(\v)$ on $\v$.
Indeed, denoting by $a_t=\exp(tA)$, we can write for every $u\in Q$:
$$ [A^\ast, \xi^\ast]_u=- \left.\frac{d}{dt}\right|_{t=0}{R_{a_t}}_*(\xi^*_{ua_t^{-1}})=-\left.\frac{d}{dt}\right|_{t=0}{R_{a_t}}_*(\widetilde{ua_t^{-1}\xi})_{ua_t^{-1}}=-\left.\frac{d}{dt}\right|_{t=0}(\widetilde{ua_t^{-1}\xi})_{u}=(A.\xi)^*_u.$$

We will now show that $\theta$ satisfies a structure equation similar to the one relating the torsion of a linear connection with the connection and soldering forms. 

\begin{elem} The $2$-form $\Theta \in \Omega^2(Q, \v)$ defined at any point $u\in Q$ by 
\beq\label{Theta}\Theta_u(U,V):=2u^{-1}(\tau({\pi_M}_*U,{\pi_M}_*V)^\V),\qquad\forall\ U,V\in\T_uQ\eeq
is related to the connection form $\alpha$ and the soldering-like form $\theta$ by
\beq\label{str2}
\Theta = \d\theta + \alpha \wedge \theta \ .
\eeq
where $(\alpha \wedge \theta)(U,V) := \alpha(U) .\theta(V)- \alpha(V) .\theta(U)$.
\end{elem}

\begin{proof}
By bilinearity and tensoriality, it suffices to prove \eqref{str2} when each entry is either a horizontal lift of a section of $\H M$ or $\V M$, or a fundamental vertical vector field. Let $X,Y$ be sections of $\H M$, $Z,W$ sections of $\V M$ and $A,B\in \k$. We denote by $\hat Z$ the $\v$-valued function on $Q$ induced by the section $Z$ of $\V M$ as in \eqref{hut}. Since $\theta(\tilde X)=\theta(\tilde Y)=\theta(A^*)=\theta(B^*)=0$, and $\theta(\tilde Z)=\hat Z$, $\theta(\tilde W)=\hat W$, we have for every $u\in Q$:
\bea(\d\theta + \alpha \wedge \theta)_u(\tilde X,\tilde Y)&=&\d\theta_u(\tilde X,\tilde Y)=-\theta_u([\tilde X,\tilde Y])=-u^{-1}([X,Y]^\V)\\
&=&-u^{-1}((\nabla^{g^M}_XY-\nabla^{g^M}_YX)^\V)=-u^{-1}((\nabla^\tau_XY-\nabla^\tau_YX-2\tau(X,Y))^\V)\\
&=&u^{-1}((2\tau(X,Y))^\V)=\Theta_u(\tilde X,\tilde Y)\ ,
\eea
as $(\nabla^\tau_XY)^\V=(\nabla^\tau_YX)^\V=0$. Using \eqref{nablatauv}, we obtain:
\bea(\d\theta + \alpha \wedge \theta)_u(\tilde X,\tilde Z)&=&\d\theta_u(\tilde X,\tilde Z)=\tilde X(\theta(\tilde Z))_u-\theta_u([\tilde X,\tilde Z])=\tilde X(\hat Z)(u)-u^{-1}([X,Z]^\V)\\
&=&u^{-1}(\nabla^\tau_XZ-(\nabla^{g^M}_XZ-\nabla^{g^M}_ZX)^\V)\\&=&-u^{-1}(\nabla^\tau_XZ-(\nabla^\tau_XZ-\nabla^\tau_ZX-2\tau(X,Z))^\V)=0=\Theta_u(\tilde X,\tilde Z)\ ,
\eea
as $\tau(X,Z)^\V=0$ by Lemma \ref{tors}. Similarly, using that $\V M$ is integrable, we can write:
\bea(\d\theta + \alpha \wedge \theta)_u(\tilde W,\tilde Z)&=&\d\theta_u(\tilde W,\tilde Z)=\tilde W(\theta(\tilde Z))_u-\tilde Z(\theta(\tilde W))_u-\theta_u([\tilde W,\tilde Z])\\&=&\tilde W(\hat Z)(u)-\tilde Z(\hat W)(u)-u^{-1}([W,Z]^\V)\\
&=&u^{-1}(\nabla^\tau_WZ-\nabla^\tau_ZW-[W,Z])=u^{-1}(2\tau(W,Z))=\Theta_u(\tilde W,\tilde Z)\ ,
\eea

Since $[A^*,\tilde X]=0$ by the right-invariance of the horizontal distribution, we get
\bea(\d\theta + \alpha \wedge \theta)_u(\tilde X,A^*)&=&\d\theta_u(\tilde X,A^*)=-\theta_u([\tilde X,A^*])=0=\Theta_u(\tilde X,A^*).
\eea

Next, denoting by $a_t=\exp(tA)$, we have 
$$A^*(\theta(\tilde Z))_u=\left.\frac{d}{dt}\right|_{t=0}(\theta(\tilde Z)_{ua_t})=\left.\frac{d}{dt}\right|_{t=0}((ua_t)^{-1}Z)=-A.u^{-1}Z\ ,$$
whence
\bea(\d\theta + \alpha \wedge \theta)_u(\tilde Z,A^*)&=&\d\theta_u(\tilde Z,A^*)-\alpha(A^*).\theta_u(\tilde Z)=-A^*(\theta(\tilde Z))_u-A.u^{-1}Z=0\\&=&\Theta_u(\tilde Z,A^*)\ .
\eea

Finally, $(\d\theta + \alpha \wedge \theta)_u(A^*,B^*)=0=\Theta_u(A^*,B^*)$, thus proving the lemma.
\end{proof}

Equations \eqref{str1} and \eqref{str2} are 
called {\it first and second structure equations} of $Q$.

Our next aim is to define a Lie algebra structure on the vector space $\gg := \k \oplus \v $ induced from the Lie algebra structure on the space of vector fields on $Q$ by the
injective map $\Phi: \gg = \k \oplus \v  \rightarrow \Gamma(\T Q), \ A + \xi \mapsto A^\ast + \xi^\ast$, for $A \in \k$ and $\xi$ in $\v $. This can be viewed as an extension to non-homogeneous setting of the classical Nomizu construction for connections with parallel curvature and torsion \cite{nom}.

\begin{elem} The image of the map $\Phi$ is closed under the bracket of vector fields. 
\end{elem}
\begin{proof}
Since by \eqref{br} and \eqref{br1}, $[A^\ast, B^\ast] = [A, B]^\ast$ for every $A, B \in \k$ and $[A^\ast, \xi^\ast] = (A. \xi)^\ast$ for every $A \in \k$ and $\xi \in \v $,
it only remains to consider the bracket of the standard horizontal vector fields induced by $\xi_1, \xi_2 \in \v $. 

For every $\xi_1,\xi_2\in\v$, the first structure equation \eqref{str1}, together with \eqref{ral} applied to the vector bundle $\V M$ gives
\be\label{constant3} \alpha_u([\xi^\ast_1, \xi^\ast_2]) =- \Omega^\alpha_u (\xi_1^\ast, \xi_2^\ast)=-u^{-1}\circ R^\alpha_{u\xi_1,u\xi_2}\circ u\ ,
\ee
where $R^\alpha$ is the curvature operator of $\nabla^\alpha$, viewed as $2$-form on $M$ with values in $\Lambda^2\V M$. Since $\V M$ is totally geodesic with respect to $\nabla^\tau$ and $\nabla^\alpha$ is the restriction of $\nabla^\tau$ to $\V M$, we thus get that 
\be\label{ralp}R^\alpha=\pi_{\Lambda^2\V M}\circ R^\tau,\ee 
which is a parallel tensor by Proposition \ref{rparallel1}. 

This shows that $u\mapsto u^{-1}\circ R^\alpha_{u\xi_1,u\xi_2}\circ u$ is  constant along horizontal curves in $Q$, thus constant on $Q$ since every two points of $Q$ can be joined by a horizontal curve.
Consequently, there exists an element $R\in\Lambda^2\v\otimes \k\subset \Lambda^2\v\otimes\so(\v)$ such that 
\be\label{rco1}
u^{-1}\circ R^\alpha_{u\xi_1,u\xi_2}\circ u= R(\xi_1, \xi_2)
 \qquad \forall\ \xi_1,\xi_2\in\v,\ \forall\ u\in Q\ .
\ee

Moreover, the second structure equation \eqref{str2}, together with the fact that $\alpha(\xi_i^*)=0$ and $\theta(\xi_i^*)=\xi_i$ are constant functions for $i=1,2$, implies:
\be\label{constant1}
\theta_u([\xi^\ast_1, \xi^\ast_2])=-\d\theta_u(\xi^\ast_1, \xi^\ast_2)= - \Theta_u(\xi^\ast_1, \xi_2^\ast)=-2u^{-1}\tau(u\xi_1,u\xi_2)\ .
\ee

Like before, since $\tau$ is $\nabla^\tau$-parallel, there exists an element $T \in \Lambda^3\v $ such that 
\be\label{thc}
u^{-1}\tau(u\xi_1,u\xi_2)= T(\xi_1, \xi_2),\qquad\forall\ \xi_1,\xi_2\in\v,\ \forall\ u\in Q\ .
\ee

The above equations \eqref{constant3}--\eqref{thc} yield
\be\label{constant2}
[\xi^\ast_1, \xi^\ast_2] =-2T(\xi_1, \xi_2)^\ast- R(\xi_1, \xi_2)^\ast \ ,
\ee
thus proving the lemma.
\end{proof}

Hence we  can define a Lie bracket on $\gg = \k \oplus \v $ extending the one on $\k$ by
\be\label{las}
[A, \xi] :=   A. \xi \qquad \mbox{and}\qquad  [\xi_1, \xi_2]   := -2T(\xi_1, \xi_2) - R(\xi_1, \xi_2) \ ,
\ee
which satisfies the Jacobi identity since the map $\Phi:\gg\to \Gamma(\T Q)$ is injective. In particular, the following relation, analogous to \eqref{br} and \eqref{br1} holds for every $\xi_1,\xi_2\in\v$:
\be\label{br2} [\xi^\ast_1, \xi^\ast_2] =[\xi_1, \xi_2]^*\ .\qquad 
\ee

The metric $g^M$ restricted to $\V M$ induces a $\k$-invariant scalar product $\langle\cdot,\cdot\rangle_\v$ on $\v$. For any element $B=A+\xi\in\gg$ we denote by $B_\v:=\xi$ its component in $\v$. Since $T$ is skew-symmetric, Equation \eqref{las} yields for every $\xi_1,\xi_2,\xi_3\in \v $:
\be\label{natr}\langle [\xi_1,\xi_2]_\v,\xi_3\rangle_\v+\langle \xi_2,[\xi_1,\xi_3]_\v\rangle_\v=0\ .\ee

The decomposition $\T Q=\T^{hor} Q\oplus \T^\k Q$ of the tangent bundle of $Q$ given by the connection $\alpha$ can be refined as
$$\T Q=\T^\hh Q\oplus \T^\v Q\oplus \T^\k Q\ ,$$
where $\T^\hh Q_u=\{\tilde X_u\ |\ X\in \H M_{\pi_M(u)}\}$, $\T^\v Q_u=\{\xi^*_u\ |\ \xi\in \v\}$ and $\T^\k Q_u=\{\A^*_u\ |\ A\in \k\}$.

The map $\Phi : \gg \rightarrow \Gamma(\T Q)$ is by definition of the  Lie algebra structure on $\gg$ 
a Lie algebra homomorphism, i.e. it defines on $Q$ an infinitesimal $\gg$-principal bundle structure  over some locally defined manifold $N$, whose fibers are the leaves of the integrable distribution $\Phi(\gg)=\T^\v Q\oplus \T^\k Q$ of $Q$. Since $(\pi_M)_*^{-1}(\V M)=\T^\v Q\oplus \T^\k Q$, this locally defined manifold $N$ is the same as the basis of the standard submersion $\pi:M\to N$ introduced in the previous section. 

\begin{ere}\label{defp}
In order to distinguish the two principal bundles over $M$ and $N$ having the same total space $Q$, we will refer from now on to the newly constructed  infinitesimal $\gg$-principal bundle as $\pi_N:P\to N$. As differentiable manifolds, we thus have $P=Q$, and the three submersions $\pi_M,\pi_N$ and $\pi$, are related by $\pi_N=\pi\circ\pi_M$.
\end{ere}

Lemma \ref{tot-geod} shows that the metric of $M$ projects to a metric $g^N$ on $N$ with Levi-Civita covariant derivative denoted by $\nabla^{g^N}$. Moreover,
Lemma \ref{projectable} shows that the horizontal part $\tau^\hh$ of $\tau$ projects to a $3$-form $\sigma$ on $N$ defining a covariant derivative $\nabla^\sigma:=\nabla^{g^N}+\sigma$ with parallel skew-symmetric torsion. 

We will now introduce a connection on the principal bundle $P$ over $N$. 

\begin{elem}\label{defc}
The $1$-form $\gamma:= \alpha + \theta  \in \Omega^1(P, \gg)$ is a connection form on $P$ with respect to the infinitesimal
$\gg$--principal bundle structure, i.e. it satisfies  $\gamma(B^*)=B$ for every $B\in \gg$ and 
\be\label{binv}
(\LieD_{B^*} \gamma )(U) =  - [B, \gamma(U)],\qquad \forall\ B\in \gg,\ \forall \ U\in\Gamma(\T P)\ .
\ee
\end{elem}
\proof
The relation $\gamma(B^*)=B$ is tautological from the definition of the infinitesimal action of $\gg$ on $P$. Indeed, if $B=A+\xi\in \k\oplus \v $ then 
$$\gamma(B^*)=(\alpha+\theta)(A^*+\xi^*)=\alpha(A^*)+\theta(\xi^*)=A+\xi=B.$$

By definition, the connection 1-form $\alpha$ of the $K$-principal bundle $Q$ over $M$ is $K$-equivariant, i.e. $R_a^*\alpha=\Ad_{a^{-1}}(\alpha)$ for every $a\in K$. Differentiating this at the identity we get
\be\label{ainv} (\LieD_{A^*} \alpha )(U) =  - [A, \alpha(U)],\qquad \forall\ A\in \k,\ \forall \ U\in\Gamma(\T P)\ .
\ee

Moreover, since $A^*\lrcorner \theta=0$ and $A^*\lrcorner \Theta=0$ for every $A\in\k$, the Cartan formula together with \eqref{str2} yields $ \LieD_{A^*} \theta= A^*\lrcorner \d\theta=-A^*\lrcorner(\alpha\wedge \theta)=-A.\theta$, whence
\be\label{tinv} (\LieD_{A^*} \theta)(U) = -A.\theta(U)=-[A,\theta(U)],\qquad \forall\ A\in \k,\ \forall \ U\in\Gamma(\T P)\ .
\ee

Equation \eqref{binv} follows from \eqref{ainv} and \eqref{tinv} when $B\in\k$. It remains to check it when $B=\xi$ is a vector in $\v $. By tensoriality and linearity, it is sufficient to consider three cases: when $U=A^*$ for some $A\in \k$, $U=\eta^*$ for $\eta\in \v $ and $U=\tilde X$ for $X\in\Gamma(\H M)$. Since $\gamma(U)$ is constant in each of these cases, we get (using the structure equations \eqref{str1} and \eqref{str2} in the last case):
\bea(\LieD_{\xi^*} \gamma )(A^*)&=&-\gamma([\xi^*,A^*])=-\gamma([\xi,A]^*)=-[\xi,A]=-[\xi,\gamma(A^*)]\\
(\LieD_{\xi^*} \gamma )({\eta}^*)&=&-\gamma([\xi^*,{\eta}^*])=-\gamma([\xi,{\eta}]^*)=-[\xi,\eta]=-[\xi,\gamma(\eta^*)]\\
(\LieD_{\xi^*} \gamma )(\tilde X)&=&-\gamma([\xi^*,\tilde X])=\d\gamma(\xi^*,\tilde X)=(\Omega^\alpha-\tfrac12\alpha\wedge\alpha+\Theta-\alpha\wedge\theta)(\xi^*,\tilde X)\\
&=&(\Omega^\alpha+\Theta)(\xi^*,\tilde X)\ .
\eea
As $\gamma(\tilde X)=0$, we have $[\xi,\gamma(\tilde X) ]=0$, and it remains to check that $(\Omega^\alpha+\Theta)(\xi^*,\tilde X)=0$.

For every $u\in Q$ we denote by $W:=u(\xi)\in \V M$. Then by definition $\tilde W_u=\xi^*_u$, so using \eqref{ral} and \eqref{Theta} we infer
$$\Omega^\alpha(\xi^*,\tilde X)_u+\Theta(\xi^*,\tilde X)_u=u^{-1}\circ R^\alpha_{W,X}\circ u+2u^{-1}\tau(W,X)^\V\ .$$
Since $X\in \H M$ and $W\in \V M$, the right hand side term vanishes by \eqref{ralp}, the pair symmetry of $R^\tau$ (Lemma \ref{symm}) and the fact that $\tau$ has no component in $\Lambda^2\V M\otimes \H M$ (Lemma \ref{tors}).
\qed

\medskip

As $\gamma = \alpha + \theta$, \eqref{str1} and \eqref{str2} yield  the following decomposition of the curvature form $\Omega^\gamma$:
\be\label{oba}
\Omega^\gamma = \d\gamma+\tfrac12\gamma\wedge\gamma= (\d\alpha + \tfrac12  \alpha \wedge \alpha) + (\d\theta + \alpha \wedge \theta)+
\tfrac12 (\theta \wedge \theta)
=\Omega^\alpha+\Theta+\tfrac12 (\theta \wedge \theta) \ .
\ee

\subsection{Geometries with parallel curvature}\label{reduc}
We will now show that the curvature $R^\gamma$ of the connection $\gamma$ on the principal bundle $P$ over $N$, viewed as a 2-form with values in the adjoint bundle $\ad(P)$, is parallel with respect to the tensor product connection $\nabla^\sigma\otimes \nabla^\gamma$.

\begin{epr} \label{corrpa}
The section $R^\gamma$ of $\Lambda^2 \T N\otimes\ad(P)$ is parallel with respect to $\nabla^\sigma\otimes\nabla^\gamma$.
\end{epr}

\begin{proof} For every vector fields $Y,Z\in \Gamma(\T N)$ \eqref{ral} reads
\be\label{sts}u\Omega^\gamma(\tilde Y,\tilde Z)=R^\gamma(Y,Z),\qquad\forall u\in P\ ,
\ee
where here $u\in P$ is seen as an isomorphism from $\gg$ to the fibre of $\ad(P)$ at $\pi_N(u)\in N$.

Using  \eqref{nablatauv} and \eqref{sts} we thus get for every vector fields $X,Y,Z$ on $N$ and $u\in P$:
\beaa( (\nabla^\sigma\otimes\nabla^\gamma)_XR^\gamma)(Y,Z)&=&\nabla^\gamma_X(R^\gamma(Y,Z))-R^\gamma(\nabla^\sigma_XY,Z)-R^\gamma(Y,\nabla^\sigma_XZ)\label{rpa}\\
&=&u\left(\tilde X(\Omega^\gamma(\tilde Y,\tilde Z))-\Omega^\gamma(\widetilde{\nabla^\sigma_X Y},\tilde Z)-\Omega^\gamma(\tilde Y,\widetilde{\nabla^\sigma_X Z}) \right)\nonumber\ .
\eeaa

For every vector field $X$ on $N$ we denote by $X_M$ its horizontal lift to a vector field on $M$ with respect to the standard submersion $\pi:M\to N$. It is clear that
the horizontal lifts of $X$ to $(P,\gamma)$ and of $X_M$ to $(Q,\alpha)$ coincide. Thus we may write
$\tilde X_M = \tilde X$ for this horizontal lift. 

From \eqref{oba}, we have $\Omega^\gamma=\Omega^\alpha+\Theta+\tfrac12\theta\wedge\theta.$ For every  vector fields $X,Y$ on $N$ we thus have by \eqref{Theta} and  \eqref{ral}:
\be\label{ot}\Omega^\gamma(\tilde X,\tilde Y)=\Omega^\alpha(\tilde X,\tilde Y)+\Theta(\tilde X,\tilde Y)\ .
\ee

On the other hand, $R^\alpha$ and $\tau^m$ are $\nabla^\tau$ parallel tensors on $M$, so they define $\nabla^\tau\otimes\nabla^\alpha$-parallel 2-forms with values in $\ad(Q)$ and $\V M$ respectively. Moreover it follows from O'Neill's formulas that
$
\nabla^\tau_{X_M} Y_M =  (\nabla^\sigma_{X} Y)_M
$.
We thus get for every vector fields $X,Y,Z$ on $N$:
\bea 0&=&( (\nabla^\tau\otimes\nabla^\alpha)_{X_M}R^\alpha)({Y_M},{Z_M})\\&=&\nabla^\tau_{X_M}(R^\alpha({Y_M},{Z_M}))-R^\alpha(\nabla^\tau_{X_M}{Y_M},{Z_M})-R^\alpha({Y_M},\nabla^\tau_{X_M}{Z_M})\\
&=&\nabla^\tau_{X_M}(R^\alpha({Y_M},{Z_M}))-R^\alpha((\nabla^\sigma_{X}{Y})_M,{Z_M})-R^\alpha({Y_M},(\nabla^\sigma_{X}{Z})_M)
\eea
and 
\bea 0&=&( (\nabla^\tau\otimes\nabla^\alpha)_{X_M}\tau^m)({Y_M},{Z_M})\\&=&\nabla^\tau_{X_M}(\tau^m({Y_M},{Z_M}))-\tau^m(\nabla^\tau_{X_M}{Y_M},{Z_M})-\tau^m({Y_M},\nabla^\tau_{X_M}{Z_M})\\
&=&\nabla^\tau_{X_M}(\tau^m({Y_M},{Z_M}))-\tau^m((\nabla^\sigma_{X}{Y})_M,{Z_M})-\tau^m({Y_M},(\nabla^\sigma_{X}{Z})_M)\ .
\eea

As before, the $\gg$-valued function on $Q$ corresponding to the section $R^\alpha({Y_M},{Z_M})$ of $\ad(Q)$ is $\Omega^\alpha(\tilde X,\tilde Y)$, 
and  the $\v$-valued function on $Q$ corresponding to the section $\tau^m({Y_M},{Z_M})$ of $\V M$ is $\tfrac12\Theta(\tilde X,\tilde Y)$,
so by  \eqref{nablatauv}, the previous equations read:
\beaa 0&=&\tilde X(\Omega^\alpha(\tilde Y,\tilde Z))-\Omega^\alpha(\widetilde{\nabla^\sigma_X Y},\tilde Z)-\Omega^\alpha(\tilde Y,\widetilde{\nabla^\sigma_X Z}) \label{1e}\ ,\\
0&=&\tilde X(\Theta(\tilde Y,\tilde Z))-\Theta(\widetilde{\nabla^\sigma_X Y},\tilde Z)-\Theta(\tilde Y,\widetilde{\nabla^\sigma_X Z}) \label{2e}\ .
\eeaa

Clearly \eqref{ot}, \eqref{1e} and \eqref{2e} imply that the right hand side of  \eqref{rpa} vanishes, thus proving the lemma.
\end{proof}

\begin{elem} \label{crel} For every $u\in P$, $\xi_1,\xi_2\in \v $, and $X,Y \in \T N$, the following relation holds:
\be\label{1s2}g^N (R^\gamma_{u\xi_2}(X),R^\gamma_{u\xi_1}(Y))-g^N( R^\gamma_{u\xi_2}(Y),R^\gamma_{u\xi_1}(X))+\langle[u^{-1}R^\gamma_{X,Y},\xi_2]_\v,\xi_1\rangle_\v=0\ ,\ee
where $R^\gamma_{u\xi}$ is the endomorphism of $\T N$ defined  by 
\be
\la (u^{-1}R^\gamma_{X,Y})_\v, \xi \ra_\v = :g^N(R^\gamma_{u\xi} (X), Y),\qquad\forall\ X,Y\in\T N\ .
\ee

\end{elem}
\begin{proof}
By \eqref{oba} and \eqref{sts}, for every $X,Y \in \T N$ with horizontal lifts $X_M,Y_M\in\H M$, we can write:
\bea g^N(R^\gamma_{u\xi} (X), Y)&=&\la (u^{-1}R^\gamma_{X,Y})_\v, \xi \ra_\v=\la\Omega^\gamma(\tilde X,\tilde Y)_\v,\xi\ra_\v
=\la\Theta(\tilde X,\tilde Y)_\v,\xi\ra_\v\\&=&\la 2u^{-1}\tau(X_M,Y_M)^\V,\xi\ra_\v=2g^M(\tau(X_M,Y_M),u\xi)=2g^M(\tau_{u\xi}X_M,Y_M)\ ,
\eea
whence $(R^\gamma_{u\xi}X)_M=2\tau_{u\xi}X_M$. Moreover, by \eqref{thc} and \eqref{las}, we get
\bea \langle[u^{-1}R^\gamma_{X,Y},\xi_2]_\v,\xi_1\rangle_\v &=&\la[\Omega^\gamma_u(\tilde X,\tilde Y),\xi_2]_\v,\xi_1\ra_\v=\la\Omega^\alpha_u(\tilde X,\tilde Y).\xi_2,\xi_1\ra_\v+\la[\Theta_u(\tilde X,\tilde Y),\xi_2]_\v,\xi_1\ra_\v\\
&=&g^M(R^\alpha_{X_M,Y_M}u\xi_2,u\xi_1)-2\la u^{-1}\tau(u\Theta_u(\tilde X,\tilde Y),u\xi_2),\xi_1\ra_\v\\
&=&g^M(R^\tau_{X_M,Y_M}u\xi_2,u\xi_1)-4g^M(\tau(\tau(X_M,Y_M)^\V,u\xi_2),u\xi_1)\ .
\eea

Denoting the vectors $u\xi_1,u\xi_2$ in $\V M$ by $V,W$ respectively, \eqref{1s2} is thus equivalent to 
$$4(g^M(\tau_WX_M,\tau_VY_M)-g^M(\tau_WY_M,\tau_VX_M))+g^M(R^\alpha_{X_M,Y_M}W,V)-4g^M(\tau_WV,\tau_{X_M}Y_M)=0\ ,$$
which follows directly from \eqref{rt}.
\end{proof}

In view of the above results it makes sense to introduce the following:

\begin{ede} \label{pgwt}A {\em geometry with parallel curvature} $(N,g^N,\sigma,P,\gg,\gamma,\k,\v,\la\cdot,\cdot\ra_\v)$ is defined by
a Riemannian manifold $(N, g^N)$ with Levi-Civita covariant derivative $\nabla^{g^N}$, carrying a metric covariant derivative $\nabla^\sigma:=\nabla^{g^N}+\sigma$ with parallel skew-symmetric torsion $T^\sigma=2\sigma$, and a (locally defined) $G$-principal bundle $\pi_N : P \rightarrow N$ endowed with a connection form $\gamma \in \Omega^1(P, \gg)$, where $\gg$ is the Lie algebra of $G$, such that the following properties hold:
\begin{itemize}
\item[$(i)$] If $\nabla^\gamma$ denotes the covariant derivative induced by $\gamma$ on $\ad(P)$, then the section $R^\gamma$  of $\Lambda^2 \T N\otimes\ad(P)$ defined by the curvature form $\Omega^\gamma$ of $\gamma$ is parallel with respect to $\nabla^\sigma\otimes\nabla^\gamma$;
\item[$(ii)$] There exists a direct sum decomposition $\gg=\k\oplus \v$, where $\k\subset\gg$ is a Lie sub-algebra of compact type, and $\v$ is a faithful $\k$-representation with respect to the adjoint action, which carries a $\k$-invariant scalar 
product $\langle\cdot,\cdot\rangle_\v$ such that the splitting $\gg = \k \oplus \v $ is naturally reductive, i.e. $\langle [\xi_1,\xi_2]_\v,\xi_3\rangle_\v+\langle \xi_2,[\xi_1,\xi_3]_\v\rangle_\v=0$ for every $\xi_1,\xi_2,\xi_3\in \v $;
\item[$(iii)$]  For every $u\in P$, $\xi_1,\xi_2\in \v $, and $X,Y \in \T N$, the following relation holds:
\be\label{s2}g^N ([R^\gamma_{u\xi_2},R^\gamma_{u\xi_1}](X),Y)+\langle[u^{-1}R^\gamma_{X,Y},\xi_2]_\v,\xi_1\rangle_\v=0\ ,\ee
where $R^\gamma_{u\xi}$ is the skew-symmetric endomorphism of $\T N$ defined by 
\be\label{og}
\la (u^{-1}R^\gamma_{X,Y})_\v, \xi \ra_\v = :g^N(R^\gamma_{u\xi} (X), Y),\qquad\forall\ X,Y\in\T N\ .
\ee
\end{itemize}
\end{ede}

One can summarize the results of this section in the following:

\begin{ath}\label{redu}
Let $(M^n, g,\tau)$ be a geometry with parallel skew-symmetric torsion (Definition \ref{gwt}), with standard decomposition $\T M=\H M\oplus \V M$ (Definition \ref{stdec}). Then the base $N$ of the standard submersion (Definition \ref{stsub}) carries a geometry with parallel curvature (Definition \ref{pgwt}) canonically induced by the geometry of $M$.
\end{ath}

\begin{proof} Consider the (locally defined) standard submersion $\pi:M\to N$. From Lemma \ref{tot-geod}, there exists a unique Riemannian metric $g^N$ on $N$ making $\pi$ into a Riemannian submersion. Lemma \ref{projectable} and Remark \ref{r411} show that there exists a unique $3$-form $\sigma$ on $N$ such that $\pi^*\sigma=\tau^\hh$, and the connection $\nabla^\sigma=\nabla^{g^N}+\sigma$ has parallel skew-symmetric torsion $T^\sigma=2\sigma$. 
The Lie algebra $\gg$, the infinitesimal $\gg$-principal bundle $P$ over $N$, the connection $\gamma$ on $P$ were constructed in \eqref{las}, Remark \ref{defp}, and Lemma \ref{defc} respectively. The properties $(i)-(iii)$ from Definition \ref{pgwt} follow from Proposition \ref{corrpa}, Equation \eqref{natr} and Lemma \ref{crel} respectively.
\end{proof}
\begin{ere} \label{nrp}
A geometry with parallel curvature with the base manifold a point is nothing but a naturally reductive decomposition $(\gg=\k\oplus\v,\la\cdot,\cdot\ra_\v)$ such that the adjoint action of $\k$ on $\v$ is faithful. 
\end{ere}

\section{The inverse construction}\label{inverse}

The aim of this section is to prove the following converse of Theorem \ref{redu}:

\begin{ath}\label{redu1}
Let $(N,g^N,\sigma,P,\gg,\gamma,\k,\v,\la\cdot,\cdot\ra_\v)$ be a geometry with parallel curvature, and let $\T^{\k} P$ be the integrable distribution of $\T P$ spanned at each point by fundamental vertical vector fields $A^*$ with $A\in{\k}$. Then the manifold $M$, locally defined as the space of leaves of $\T^{\k} P$, carries a geometry with parallel skew-symmetric torsion $(g,\tau)$. 
\end{ath}

\begin{proof}
Let us start by deriving a formula which will be necessary later on. The fact that $R^\gamma$ is parallel with respect to $\nabla^\sigma\otimes\nabla^\gamma$, together with \eqref{rpa}, and the Bianchi identity
$\d \Omega^\gamma = - \gamma \wedge \Omega^\gamma$, shows that for all vector fields $X, Y, Z$ on $N$ with horizontal lifts $\tilde X,\tilde Y, \tilde Z$ to $P$, we have
\bea
0 &=& \mathop{\mathfrak{S}}_{XYZ} 
\left(\tilde X(\Omega^\gamma(\tilde Y, \tilde Z)) - \Omega^\gamma( \widetilde{[X,Y]}, \tilde Z)  \right) \\
&=&
\mathop{\mathfrak{S}}_{XYZ} 
\left(\Omega^\gamma( \widetilde{\nabla^\sigma_XY}, \tilde Z )
 + \Omega^\gamma(\tilde Y, \widetilde{\nabla^\sigma_XZ})
   - \Omega^\gamma(\widetilde{[X,Y]}, \tilde Z  \right) \\
&=&
\mathop{\mathfrak{S}}_{XYZ} 
\left(\Omega^\gamma( \widetilde{\nabla^\sigma_XY}, \tilde Z )
 + \Omega^\gamma(\tilde Z, \widetilde{\nabla^\sigma_YX})
   - \Omega^\gamma(\widetilde{[X,Y]}, \tilde Z  \right) \\
&=&
2  \mathop{\mathfrak{S}}_{XYZ}   
\Omega^\gamma(\widetilde{\sigma(X,Y)}, \tilde Z) \ , 
\eea
as
$
\nabla^\sigma_XY - \nabla^\sigma_YX -[X,Y] = \nabla^{g^N}_XY - \nabla^{g^N}_YX - [X,Y] + 2 \sigma(X,Y)
= 2 \sigma(X,Y)
$.
We thus obtain 
\begin{equation}\label{s1}
 \mathop{\mathfrak{S}}_{XYZ} \Omega^\gamma (\widetilde{\sigma(X, Y)}, \tilde Z)= 0,\qquad\forall\ X,Y,Z\in \T N\ .
 \end{equation}

\medskip

{\bf Step 1.} We choose any scalar product $\la\cdot,\cdot\ra$ on $\gg$ which extends $\la\cdot,\cdot\ra_\v$ and makes $\k$ and $\v$ orthogonal, and define a Riemannian metric $g^P$ on the total space of $P$ by 
\be\label{met}
g^P(U,V):= ((\pi_N)^* g^N)(U,V) + \langle\gamma(U), \gamma(V)\rangle\ .
\ee
In this way, the projection $\pi_N:P \rightarrow N$ becomes a Riemannian submersion. 
The tangent bundle $\T P$ splits into a $g^P$-orthogonal direct sum of distributions $\T P=\T^{hor} P\oplus \T^\v P\oplus\T^{\k} P$, where $\T^{hor} P:=\ker(\gamma)$ is spanned at each point by horizontal lifts $\tilde X$ of vector fields $X$ on $N$, and $\T^\v P$ and $\T^{\k} P$ are spanned at each point by fundamental vertical vector fields $A^*$ with $A\in\v$ and $A\in{\k}$ respectively.
The Levi-Civita connection of $g^P$ can be easily computed using these adapted vector fields. By definition,
$$g^P(\tilde X,\tilde Y)=g^N(X,Y),\qquad g^P(\tilde X,A^*)=0,\qquad g^P(A^*,B^*)=\langle A,B\rangle\ .$$
Moreover, since $\gamma([\tilde X,\tilde Y])=-\d\gamma(\tilde X,\tilde Y)=-\Omega^\gamma(\tilde X,\tilde Y)$, we obtain
$$[A^*, B^*] = [A, B]^*, \qquad [A^*, \tilde X] = 0,\qquad [\tilde X, \tilde Y] = \widetilde{[X, Y]} - \Omega^\gamma(\tilde X, \tilde Y)^*\ .$$

The Koszul formula immediately implies that the Levi-Civita connection $\nabla^{g^P}$ of $g^P$ is given by
\beaa
\nabla^{g^P}_{\tilde X} \tilde Y &=&  \widetilde{\nabla^{g^N}_XY} - \tfrac12  \Omega^\gamma(\tilde X, \tilde Y)^*\label{k1}\\
\nabla^{g^P}_{\tilde X} A^*   &=& \nabla^{g^P}_{A^*}\tilde X  = \tfrac12  \Omega^\gamma_A(\tilde X ) \label{k2}\\[.5ex]
g^P(\nabla^{g^P}_{A^*} B^*, \tilde X) &=& 0\label{k3}\\
g^P(\nabla^{g^P}_{A^*} B^*, C^*) &=& \tfrac12  \left(   \langle [A, B], C\rangle - \langle[B, C], A\rangle + \langle[C, A], B\rangle  \right)\label{k4}
\eeaa
where $\nabla^{g^N}$ denotes the Levi-Civita covariant derivative of $(N,g^N)$ and $\Omega^\gamma_A(\tilde X )$ is the horizontal vector field of $P$ satisfying
$
g^P(\Omega^\gamma_A(\tilde X ), \tilde Y) = \langle\Omega^\gamma(\tilde X, \tilde Y), A\rangle
$
for each vector field $Y\in \Gamma(\T N)$. When $A\in\v$, $\Omega^\gamma_A$ is related to the endomorphism $R^\gamma_{uA}$ defined in \eqref{og} by $\Omega^\gamma_A(\tilde X_u)=\widetilde{R^\gamma_{uA}(X)}$.

{\bf Step 2.} We show that the metric $g^P$ projects to a metric $g^M$ on $M$ making the (locally defined) projection $\pi_M:(P,g^P)\to (M,g^M)$ a Riemannian submersion with totally geodesic fibres tangent to $\T^{\k} P$. The distribution 
$\T^{\k} P$ is totally geodesic by \eqref{k3} and \eqref{k4}, and the fact that $[{\k},\v ]\subset \v $ and $[{\k},\k ]\subset \k $.

It remains to show that the restriction $h$ of $g^P$ to $\T^{hor} P\oplus \T^\v P$ is constant in $\T^{\k} P$-directions, that is, $(\LieD_{A^*}h)(U,V)=0$ for every $A\in {\k}$ and $U,V\in\Gamma(\T P)$. Note first that $$h(\nabla^{g^P}_UA^*,V)=U(h(A^*,V))-h(A^*,\nabla^{g^P}_UV)=0.$$
We thus obtain 
\bea(\LieD_{A^*}h)(U,V)&=&A^*(h(U,V))-h([A^*,U],V)-h(U,[A^*,V])\\&=&(\nabla^{g^P}_{A^*}h)(U,V)+h(\nabla^{g^P}_UA^*,V)+h(U,\nabla^{g^P}_VA^*)=(\nabla^{g^P}_{A^*}h)(U,V).\eea
Since $\T^{\k} P$ is totally geodesic, it is clear that this last term vanishes when $U$ or $V$ are tangent to $\T^{\k} P$. Hence, to check the vanishing of $\nabla^{g^P}_{A^*}h$, it is sufficient to consider the cases $(U,V)=(\tilde X,\tilde Y)$, $(U,V)=(\tilde X,\xi^*)$ and $(U,V)=(\xi^*,{\xi_1}^*)$, where $X,Y$ are vector fields on $N$ and $\xi,{\xi_1}\in \v $. Using \eqref{k1}--\eqref{k4} and the fact that $A^*(h(U,V))=0$ for the above chosen vector fields $(U,V)$ and $A\in{\k}$, we get:
\bea(\nabla^{g^P}_{A^*}h)(\tilde X,\tilde Y)&=&-h(\nabla^{g^P}_{A^*}\tilde X,\tilde Y)-h(\tilde X,\nabla^{g^P}_{A^*}\tilde Y)=-\tfrac12 \left( g^P(\Omega^\gamma_A(\tilde X ),\tilde Y)+g^P(\tilde X,\Omega^\gamma_A(\tilde Y ))\right)\\
&=&-\tfrac12 \left(\langle \Omega^\gamma(\tilde X ,\tilde Y),A\rangle+\langle \Omega^\gamma(\tilde Y ,\tilde X),A\rangle\right)=0\ , \\
(\nabla^{g^P}_{A^*}h)(\tilde X,\xi^*)&=&-h(\nabla^{g^P}_{A^*}\tilde X,\xi^*)-h(\tilde X,\nabla^{g^P}_{A^*}\xi^*)=0 \ ,\\
(\nabla^{g^P}_{A^*}h)(\xi^*,{\xi_1}^*)&=&-h(\nabla^{g^P}_{A^*}\xi^*,{\xi_1}^*)-h(\xi^*,\nabla^{g^P}_{A^*}{\xi_1}^*)=0\ .
\eea

This shows that there exists a Riemannian metric $g^M$ on $M$ such that $(\pi_M)^*g^M=h$. Note that by \eqref{met} we have
\be\label{eh}h(U,V)= ((\pi_N)^* g^N)(U,V) + \langle\gamma(U)_\v, \gamma(V)_\v\rangle_\v\ .
\ee

{\bf Step 3.} We define a $3$-form $\tau^P$ on $P$ which projects onto a $3$-form $\tau$ on $M$. Let $\gamma=\gamma^{\k}+\gamma^\v$ be the decomposition of the connection form $\gamma$ corresponding to the decomposition $\gg={\k}\oplus\v$.
Inspired by formulas \eqref{Theta}, \eqref{thc} and \eqref{las} in the previous section, we define 
\be\label{tau}\tau^P=\tau_1+\tau_2+\tau_3\ ,\ee
where
\begin{eqnarray}\label{ta1}\tau_1(U,V,W)&:=&(\pi_N^* \sigma)(U,V,W)\ ,\\ 
\label{ta2}\tau_2(U,V,W)&:=& \tfrac12   \mathop{\mathfrak{S}}_{UVW}  \langle\Omega^\gamma(U,V), \gamma^\v(W)\rangle  \ ,\\ 
\label{ta3}\tau_3(U,V,W)&:=&  -\tfrac12  \langle  [\gamma^\v (U), \gamma^\v (V)],  \gamma^\v (W) \rangle \ .\end{eqnarray}
Note that the $3$-form $\tau_3$ is skew-symmetric because of the natural reductivity of the decomposition $\gg={\k}\oplus \v $. The $3$-form $\tau^P$ is clearly horizontal with respect to $\pi_M$, in the sense that it vanishes whenever one of the entries belongs to $\T^{\k} P$. In order to show that it is projectable onto $M$, it suffices to show that its  Lie derivative with respect to any fundamental vector field  $A^*$  with $A \in {\k}$ vanishes. 

First, it is clear that $\tau_1$ is  projectable onto $N$, so $\LieD_{A^*}\tau_1=0$ for every $A \in \gg$. Using the equivariance of $\gamma$ we have as before 
$$(\LieD_{A^*} \gamma )(U) =  - [A, \gamma(U)],\quad (\LieD_{A^*}\Omega^\gamma )(U,V) =  - [A, \Omega^\gamma(U,V)],\quad\forall\ A\in \gg,\ \forall \ U,V\in\Gamma(\T P)\ .$$
In particular, when $A\in{\k}$, the bracket with $A$ preserves the decomposition $\gg={\k}\oplus\v$, whence 
$$(\LieD_{A^*} \gamma^\v )(U) =  - [A, \gamma^\v(U)],\qquad \forall\ A\in \k,\ \forall \ U\in\Gamma(\T P)\ .$$
Using these relations we can compute
$$ (\LieD_{A^*} \tau_2)(U,V,W)= -\tfrac12   \mathop{\mathfrak{S}}_{UVW}  \left(\langle[A,\Omega^\gamma(U,V)], \gamma^\v(W)\rangle+ \langle\Omega^\gamma(U,V), [A,\gamma^\v(W)]\rangle\right)=0$$
since $\ad_A$ is skew-symmetric on $\gg$, and finally, using the Jacobi identity, we get
\bea (\LieD_{A^*} \tau_3)(U,V,W)&=&\tfrac12  \langle  [[A,\gamma^\v (U)], \gamma^\v (V)],  \gamma^\v (W) \rangle + \tfrac12  \langle  [\gamma^\v (U), [A,\gamma^\v (V)]],  \gamma^\v (W) \rangle\\&& +\tfrac12  \langle  [\gamma^\v (U), \gamma^\v (V)],  [A,\gamma^\v (W)] \rangle\\
&=&\tfrac12  \langle  [[\gamma^\v (V),\gamma^\v (U)], A],  \gamma^\v (W) \rangle -\tfrac12  \langle  [A, [\gamma^\v (U), \gamma^\v (V)]], \gamma^\v (W)] \rangle=0 \ .
\eea

This shows the existence of a $3$-form $\tau$ on $M$ such that $\pi_M^*(\tau)=\tau^P$.

{\bf Step 4.} We check that $\nabla^\tau:=\nabla^{g^M}+\tau$ has parallel skew-symmetric torsion, where $\nabla^{g^M}$ denotes the Levi-Civita covariant derivative of  $(M, g^M)$. Let us denote by $\nabla^{\tau^P}=\nabla^{g^P}+\tau^P$. 
Since $\pi_M^*(\tau)=\tau^P$ and since $\pi_M:(P,g^P)\to (M,g^M)$ is a Riemannian submersion, we have 
$\nabla^\tau\tau=0$ if and only if $\nabla^{\tau^P}\tau^P$ vanishes whenever applied to vectors in $\T^{hor} P\oplus\T^\v P$. Since the vector fields of the form $\tilde X$ for $X\in\Gamma(\T N)$ span $\T^{hor} P$ and vector fields of the form $\xi^*$ for $\xi\in\v$ span $\T^\v P$ at each point, we will assume that each of the $4$ entries of $\nabla^{\tau^P}\tau^P$ is of one of these types. 

First, using \eqref{k1}--\eqref{k4} and \eqref{ta1}--\eqref{ta3}, we readily compute 
\beaa
\nabla^{\tau^P}_{\tilde X} \tilde Y &=&  \widetilde{\nabla^\sigma_XY} - \tfrac12  (\Omega^\gamma(\tilde X, \tilde Y)^{\k})^*\label{t1}\ ,\\
\nabla^{\tau^P}_{\tilde X} \xi^*   &=& 0 \label{t2}\ ,\\
\nabla^{\tau^P}_{\xi^*}\tilde X &=& \Omega^\gamma_\xi(\tilde X)\label{t3}\ ,\\
\nabla^{\tau^P}_{\xi_1^*} \xi_2^* &=& \tfrac12  ([\xi_1, \xi_2]^{\k})^*\label{t4}\ ,
\eeaa
where the superscript ${\k}$ in \eqref{t1} and \eqref{t4} denotes the projection from $\gg$ to ${\k}$. Now, $\tau_1$ vanishes unless all entries are in $\T^{hor} P$, $\tau_2$ vanishes unless two entries are in $\T^{hor} P$ and one is in $\T^\v P$, and $\tau_3$ vanishes unless all entries are in $\T^\v P$.
From \eqref{t1}--\eqref{t4} we see that the only possibly non-vanishing terms in $\nabla^{\tau^P}\tau_P$ on vectors of the type $\tilde X$ or $\xi^*$ are:
\bea(\nabla^{\tau^P}_{\tilde X}\tau_1)(\tilde Y_1,\tilde Y_2,\tilde Y_3)&=&(\nabla^{\sigma}_{X}\sigma)(Y_1,Y_2,Y_3)\ ,\\
(\nabla^{\tau^P}_{\xi^*}\tau_1)(\tilde Y_1,\tilde Y_2,\tilde Y_3)&=&\xi^*(\tau_1(\tilde Y_1,\tilde Y_2,\tilde Y_3))-\mathop{\mathfrak{S}}_{123}(\pi_N^*\sigma)(\Omega^\gamma_\xi(\tilde Y_1),\tilde Y_2,\tilde Y_3)\\
&=&-\mathop{\mathfrak{S}}_{123}\langle \Omega^\gamma(\tilde Y_1,\widetilde{\sigma(Y_2,Y_3)}),\xi\rangle
\ ,\\
(\nabla^{\tau^P}_{\tilde X}\tau_2)(\tilde Y,\tilde Z,\xi^*)&=&\tfrac12\left(\tilde X(\langle\Omega^\gamma(\tilde Y,\tilde Z),\xi\rangle)-\langle\Omega^\gamma(\widetilde{\nabla^\sigma_X Y},\tilde Z),\xi\rangle-\langle\Omega^\gamma(\tilde Y,\widetilde{\nabla^\sigma_X Z}),\xi\rangle   \right)\ ,\\
(\nabla^{\tau^P}_{\xi_1^*}\tau_2)(\tilde Y,\tilde Z,\xi_2^*)&=&\tfrac12\left(\xi_1^*(\langle\Omega^\gamma(\tilde Y,\tilde Z),\xi_2\rangle)-\langle\Omega^\gamma(\Omega^\gamma_{\xi_1}(\tilde Y),\tilde Z),\xi_2\rangle  -\langle\Omega^\gamma(\tilde Y,\Omega^\gamma_{\xi_1}(\tilde Z)),\xi_2\rangle \right)\\
&=&\tfrac12\left(-\langle[\xi_1,\Omega^\gamma(\tilde Y,\tilde Z)],\xi_2\rangle)+\langle\Omega^\gamma_{\xi_1}(\tilde Y),\Omega^\gamma_{\xi_2}(\tilde Z)\rangle  -\langle\Omega^\gamma_{\xi_2}(\tilde Y),\Omega^\gamma_{\xi_1}(\tilde Z)\rangle \right)\ ,\\
(\nabla^{\tau^P}_{\tilde X}\tau_3)(\xi_1^*,\xi_2^*,\xi_3^*)&=&-\tfrac12\tilde X(\langle[\xi_1,\xi_2],\xi_3\rangle)=0\ ,\\
(\nabla^{\tau^P}_{\xi^*}\tau_3)(\xi_1^*,\xi_2^*,\xi_3^*)&=&-\tfrac12\xi^*(\langle[\xi_1,\xi_2],\xi_3\rangle)=0\ .\\
\eea
The vanishing of the first four expressions follows from the assumption that $\nabla^\sigma$ has parallel torsion on $N$ and from \eqref{s1}, \eqref{rpa}, and \eqref{s2} respectively.
\end{proof}

\begin{ere} The geometry with parallel skew-symmetric torsion given by Theorem \ref{redu1} in the particular case where $N$ is reduced to a point, is the canonical homogeneous connection on the Riemannian homogeneous space defined by the naturally reductive decomposition $(\gg=\k\oplus\v,\la\cdot,\cdot\ra_\v)$ in Remark \ref{nrp}.
\end{ere}

\section{Geometries with torsion of special type}\label{gtst}

We have seen in Theorems \ref{redu} and \ref{redu1} that geometries with parallel skew-symmetric torsion can be characterized by means of geometries with parallel curvature, introduced in Definition \ref{pgwt}. The drawback of this notion is that it is rather intricate and hard to apprehend.

However, there is a particular class of geometries with parallel curvature which is much easier to define, sufficiently restrictive in order to allow a classification result (Theorem \ref{pgs} below), and at the same time sufficiently general in order to provide many examples of geometries with parallel skew-symmetric torsion $(M,g,\tau)$. These examples have the special feature that their vertical distribution $\V M$ is spanned by $\nabla^\tau$-parallel vector fields, and generalize Sasakian or 3-$(\alpha,\delta)$-Sasakian manifolds with $\delta=2\alpha$ (cf. \cite{ad}). 

More precisely, we make the following:

\begin{ede}\label{sgwt}
A {\em geometry with torsion of special type} is a geometry with parallel skew-symmetric torsion $(M,g,\tau)$ satisfying one of the following equivalent conditions:

- the summand $\V M$ in the standard decomposition (Definition \ref{stdec}) is spanned by $\nabla^\tau$-parallel vector fields;

- the infinitesimal holonomy algebra $\hol:=\hol(\nabla^\tau)$ acts trivially on $\v$.
\end{ede}

\begin{elem}\label{l73}
For every geometry with torsion of special type $(M,g,\tau)$, the horizontal part $\tau^\hh=\sum _\alpha \tau^{\hh_\alpha}$ of $\tau$ vanishes.
\end{elem}
\proof
By assumption, $\mathcal VM$ is spanned by an orthonormal frame of
$\nabla^\tau$-parallel vector fields $\xi_1, \ldots, \xi_r$. 
The torsion decomposes under the action of the holonomy group of $\nabla^\tau$ as
$$
\tau = \sum _\alpha \tau^{\hh_\alpha} + \sum_{i, \alpha} \xi_i \otimes F_{i \alpha} 
+ \sum_{ijk} c_{ijk}  \xi_i \wedge \xi_j \wedge \xi_k \ ,
$$
with $\tau^{\hh_\alpha} \in \Lambda^3 \hh_\alpha,  F_{i\alpha} \in \Lambda^2 \hh_\alpha$.  Since all 
components are $\nabla^\tau$-parallel, it follows that $c_{ijk}$ are constants.

We compute the action of $\tau_{\xi_i}$ on $\tau^\hh$. It is clear that
$F_{i \alpha}$  acts trivially on the components $\tau^{\hh_\beta}$ for every $\beta\neq\alpha$. From Lemma \ref{tauvtau} we thus obtain
$$
0=\tau_{\xi_i}\cdot\tau^\hh = \left(\sum_\alpha F_{i \alpha} + 3 \sum_{j, k}  c_{ijk} \xi_j \wedge \xi_k\right)\cdot\tau^\hh= \sum_\alpha F_{i \alpha}\cdot\tau^\hh= \sum_\alpha F_{i \alpha}\cdot\tau^{\hh_\alpha} \ .
$$
This shows that 
$
F_{i \alpha} \cdot \tau^{\hh_\alpha} = 0
$
for all $\alpha$. Note that $F_{i \alpha} \in \Lambda^2 \H_{\alpha} M$ is a $\nabla^\tau$-parallel $2$-form on the irreducible sub-bundle $\H_{\alpha} M$, so as an endomorphism it is proportional to a complex structure on $\H_{\alpha} M$. On the other hand, complex structures act injectively on
$3$-forms.

Assume that there exists some index $\beta$ with $\tau^{\hh_{\beta}}\neq0$. Then the above argument shows that $F_{i\beta}=0$ for every $i$, hence the decomposition $\T M=\H_{\beta} M\oplus (\H_{\beta} M)^\perp$ would satisfy the hypothesis of Lemma \ref{cr}, contradicting the indecomposability of $M$. This shows that $\tau^{\hh_{\alpha}}=0$ for every $\alpha$.
\qed

We now introduce, as announced earlier, the particular class of geometries with parallel curvature (Definition \ref{pgwt}) for which the sub-algebra $\k$ of $\gg$ vanishes:

\begin{ede}\label{las1} Let $G$ be a compact Lie group with Lie algebra $\gg$.
A {\em parallel $\gg$-structure} $(g^N,P,\gg,\gamma,\langle\cdot,\cdot\rangle,\psi)$ on a manifold $N$ is given by:
\begin{enumerate}
\item[$(i)$] a Riemannian metric $g^N$ on $N$;
\item[$(ii)$] a locally defined $G$-principal bundle $P\to N$ with adjoint bundle $\ad (P)$;
\item[$(iii)$] an $\ad_\gg$-invariant scalar product  $\langle\cdot,\cdot\rangle$ on $\gg$, thus inducing a scalar product also denoted by $\langle\cdot,\cdot\rangle$ on the fibers of $\ad (P)$;
\item[$(iv)$] a connection form $\gamma \in \Omega^1(P, \gg)$ whose curvature tensor $R^\gamma:\Lambda^2\T N\to \ad (P)$ is parallel with respect to the Levi-Civita connection of $g^N$ on $\Lambda^2\T N$ and the connection induced by $\gamma$ on $\ad (P)$;
\item[$(v)$] a Lie algebra bundle morphism $\psi: \ad (P)\to \Lambda^2\T N$ which is the metric adjoint of $-R^\gamma$, in the sense that 
\begin{equation}\label{epsi}g^N(\psi(\sigma),X\wedge Y)=-\langle \sigma,R^\gamma_{X,Y}\rangle,\qquad\ \forall\ X,Y\in\T N,\ \forall\ \sigma\in\ad(P)\ .
\end{equation}
\end{enumerate}
\end{ede}

\begin{ere}\label{clif} This definition is in many respects similar to the one of parallel even Clifford structures introduced in \cite[Def. 2.2]{cliff}. More precisely, a parallel rank $r$ even Clifford structure on $N$ satisfying the curvature condition in \cite[Thm. 3.6 (b)]{cliff}, which up to a factor 2 is exactly \eqref {epsi} above, defines a parallel $\so(r)$-structure on $N$ in the sense of Definition \ref{las1} after rescaling the scalar product on $\so(r)$ by a factor $2$.
\end{ere}

We will now explain the correspondence between parallel $\gg$-structures and geometries with torsion of special type.

\begin{epr}\label{p75}
The (locally defined) base $N$ of the standard submersion of a geometry with torsion of special type $(M,g,\tau)$ carries a parallel $\gg$-structure. Conversely, every parallel $\gg$-structure $(g^N,P,\gg,\gamma,\langle\cdot,\cdot\rangle,\psi)$ on a manifold $N$ induces a geometry with torsion of special type on the total space $P$.
\end{epr}
\begin{proof}
By Theorem \ref{redu}, $N$ carries a geometry with parallel curvature $(g^N,\sigma,P,\gg,\gamma,\k,\v,\langle\cdot,\cdot\rangle_\v)$ satisfying Definition \ref{pgwt}. This shows already that conditions $(i)$ and $(ii)$ in Definition \ref{las1} hold. By Definition \ref{sgwt}, we have $\k=0$, which by \eqref{defp} implies that $P$ can be identified with $M$ itself. Lemma \ref{l73} shows that $\tau^\hh=0$ on $M$, so $\sigma=0$ by Remark \ref{r411}, i.e. the connection $\nabla^\sigma$ on $N$ is the Levi-Civita connection of $g^N$. The Lie algebra $\gg$ is in this special case equal to $\v$ and the scalar product $\langle\cdot,\cdot\rangle:=\langle\cdot,\cdot\rangle_\v$ on $\gg$ is $\ad_\gg$-invariant. In particular $\gg$ is of compact type, thus proving condition $(iii)$. As for $(iv)$, it is a direct consequence of Definition \ref{pgwt} $(i)$. 

Finally, $(v)$ follows from Definition \ref{pgwt} $(iii)$. Indeed, with the notation introduced in \eqref{og}, the metric adjoint of $-R^\gamma$, denoted by $\psi: \ad (P)\to \Lambda^2\T N\simeq\so(\T N)$, defined by \eqref{epsi}, satisfies $\psi(u\xi)=-R^\gamma_{u\xi}$. Using \eqref{s2} together with the $\ad_\gg$ invariance of the scalar product on $\gg$,
we obtain for every $u\in P$, for every elements $\xi_1,\xi_2\in\v=\gg$, and for every tangent vectors $X,Y\in\T_{\pi_N(u)}N$:
\bea g^N(\psi([u\xi_1,u\xi_2]),X\wedge Y)&=&-\langle[u\xi_1,u\xi_2],R^\gamma_{X,Y}\rangle=-\langle u[\xi_1,\xi_2],R^\gamma_{X,Y}\rangle
=-\langle [\xi_1,\xi_2],u^{-1}R^\gamma_{X,Y}\rangle\\&=&\langle \xi_1,[u^{-1}R^\gamma_{X,Y},\xi_2]\rangle=g^N( [R^\gamma_{u\xi_1},R^\gamma_{u\xi_2}](X),Y)\\
&=&g^N( [R^\gamma_{u\xi_1},R^\gamma_{u\xi_2}],X\wedge Y)=g^N( [\psi(u\xi_1),\psi(u\xi_2)],X\wedge Y),
\eea
thus showing that $\psi$ is a Lie algebra bundle morphism.

Conversely, a parallel $\gg$-structure $(g^N,P,\gg,\gamma,\langle\cdot,\cdot\rangle,\psi)$ on $N$ defines in a tautological way a geometry with parallel curvature on $N$ (Definition \ref{pgwt}) with $\k=0$, $\v:=\gg$ and $\langle\cdot,\cdot\rangle_\v:=\langle\cdot,\cdot\rangle$. By Theorem \ref{redu1}, the total space of $P$ carries a geometry with parallel skew-symmetric torsion $(g,\tau)$, and the fact that $\k=0$ just means that the holonomy of $\nabla^\tau$ acts trivially on the vertical space $\v\simeq\gg$ (see Definition \ref{sgwt}). 
\end{proof}

\begin{ere}\label{rd}
The above result shows that a parallel $\gg$-structure $(g^N,P,\gg,\gamma,\langle\cdot,\cdot\rangle,\psi)$ on $N$ defines a geometry with parallel curvature $(N,g^N,\sigma:=0,P,\gg,\gamma,\k:=0,\v:=\gg,\langle\cdot,\cdot\rangle_\v:=\la\cdot,\cdot,\ra)$ in the sense of Definition \ref{pgwt}. More generally, for every sub-algebra $\k\subset \gg$, it defines a geometry with parallel curvature $(N,g^N,\sigma=0,P,\gg,\gamma,\k,\v:=\k^\perp,\langle\cdot,\cdot\rangle_\v:=\la\cdot,\cdot,\ra|_\v)$.
By Theorem \ref{redu1}, we thus see that the principal bundle $P$ of a parallel $\gg$-structure on $N$, as well as each of its quotients by subgroups of $G$, carry geometries with parallel skew-symmetric torsion. 
\end{ere}

\begin{exe} A parallel rank $r$ even Clifford structure on $N$ satisfying the curvature condition in \cite[Thm. 3.6 (b)]{cliff} determines a $S^{r-1}$-fibration $Z\to N$ whose vertical distribution belongs to the curvature constancy of $Z$. 
By Remark \ref{clif}, the parallel rank $r$ even Clifford structure also defines a parallel $\so(r)$-structure on $N$, whose quotient by $\so(r-1)$ is exactly $Z$. Remark \ref{rd} thus shows that $Z$ also carries a geometry with parallel skew-symmetric torsion. 

In particular, when $r=3$, a  parallel even Clifford structure is just a quaternion-Kähler structure on $N$, and $Z$ is its twistor space. The curvature condition  in \cite[Thm. 3.6 (b)]{cliff} is satisfied after rescaling the metric provided that $N$ has positive scalar curvature. We thus recover the well known fact that the twistor spaces of positive quaternion-Kähler manifolds carry a connection with parallel skew-symmetric torsion \cite{alex2}.
\end{exe}

\section{Parallel $\gg$-structures}\label{para}

Our next aim is to obtain the classification of parallel $\gg$-structures (and thus, according to Proposition \ref{p75}, the one of geometries with torsion of special type).

Let $(g^N,P,\gg,\gamma,\langle\cdot,\cdot\rangle,\psi)$ be a parallel $\gg$-structure on $N$. We derive first a useful identity, relating the Riemannian curvature of $(N,g)$ to the curvature form of of $(P,\gamma)$.

We denote by $\nabla^N$ and $R^N$ the Levi-Civita covariant derivative and the curvature tensor of $g^N$.   Since by Definition \ref{las1} $(iv)$, $\psi=-(R^\gamma)^*$ is $\nabla^N\otimes\nabla^\gamma$-parallel, we get for every vector fields $X,Y$ on $N$ and local section $\sigma$ of $\ad(P)$:
$$\psi(\nabla^\gamma_X \sigma)=\nabla^N_X(\psi(\sigma))\ ,$$
whence after a second covariant derivative and skew-symmetrization:
\begin{equation}\label{rgrn} \psi(R^\gamma_{X,Y} \sigma)=R^N_{X,Y}(\psi(\sigma))=[R^N_{X,Y},\psi(\sigma)]\ .
\end{equation}
Here the curvature $R^\gamma_{X,Y}$ acts on  $\sigma$ by the Lie bracket   
$R^\gamma_{X,Y}\sigma = [R^\gamma_{X,Y}, \sigma]$ of the Lie algebra bundle $\ad (P)$. Since $\psi$ is a Lie algebra bundle morphism, \eqref{rgrn} equivalently reads 
\beq\label{e50}[\psi(R^\gamma_{X,Y}),\psi(\sigma)]=[R^N_{X,Y},\psi(\sigma)]\eeq
for all tangent vectors $X,Y\in \T N$ and $\sigma\in \ad(P)$.

Note that there are several types of natural operations that one can make with parallel $\gg$-structures: products, reductions to ideals of the Lie algebra, restrictions to Riemannian factors of the manifold, or Whitney products. We will now explain these constructions in detail.

\subsection{Products of parallel $\gg$-structures}\label{product} Clearly, if $(g^{N_i},P_i,\gamma_i,\psi_i)$ are parallel $\gg_i$-structures on $N_i$ for $i=1,2$, and if $\gg$ denotes the direct sum $\gg_1\oplus\gg_2$, endowed with the direct sum scalar product, then $(g^{N_1}+g^{N_2},P_1\times P_2,\gamma_1+\gamma_2,\psi_1+\psi_2)$ is a parallel $\gg$-structure on $N_1\times N_2$, called the {\em product $\gg$-structure.}

Note that this construction also makes sense in the degenerate cases where $N_2$ is a point, or when $\gg_2=0$ (in which case we call this a $0$-structure).

\subsection{Reduction to an ideal of the Lie algebra}
Assume that $(P,\gamma,\psi)$ is a parallel $\gg$-structure on $(N,g^N)$ and that $\gg_1$ is an ideal of $\gg$. Since $\langle\cdot,\cdot\rangle$ is  $\ad_\gg$-invariant, $\gg$ is a direct sum of Lie algebras $\gg=\gg_1\oplus\gg_2$, where $\gg_2:=\gg_1^\perp$. Since everything is local, one can assume that $G=G_1\times G_2$, such that the Lie algebra of $G_i$ is $\gg_i$. 

\begin{elem} \label{g1} Let $G_1$ and $G_2$ be compact Lie groups with Lie algebras $\gg_1$ and $\gg_2$ endowed with bi-invariant scalar products, and let $G:=G_1\times G_2$, with Lie algebra $\gg=\gg_1\oplus\gg_2$ endowed with the direct sum scalar product. Then every parallel $\gg$-structure on $N$ with respect to this scalar product, induces in a canonical way parallel $\gg_1$- and $\gg_2$-structures on $N$.
\end{elem}
\begin{proof} Let $(g^N,P,\gamma,\psi)$ be a parallel $\gg$-structure on $N$. One can write the connection form $\gamma=\gamma_1+\gamma_2$ with $\gamma_i\in\Omega^1(P,\gg_i)$ for $i=1,2$. The $G$-equivariance property of $\gamma$ 
$$g^*\gamma=Ad_{g^{-1}}\gamma,\qquad\forall g\in G$$
shows that $\gamma_i$ are $G_i$-equivariant, $\gamma_1$ is $G_2$-invariant and $\gamma_2$ is $G_1$-invariant. Then $P_1:=P/G_2$  and $P_2:=P/G_1$ are $G_i$-principal bundles over $N$ and $\gamma_i$ projects to connection forms (also denoted by $\gamma_i$) on $P_i$. The adjoint bundle $\ad(P)$ is naturally identified to $\ad(P_1)\oplus \ad(P_2)$ (and this decomposition is parallel with respect to the covariant derivative induced by $\gamma)$. For every $X,Y\in \T N$ one has $R^\gamma_{X,Y}=R^{\gamma_1}_{X,Y}+R^{\gamma_2}_{X,Y}$. Denoting by $\iota$ the natural embedding of $\ad (P_1)$ into $\ad(P)$, we see that the composition $\psi_1:=\psi\circ\iota$ is a parallel Lie algebra bundle morphism from $\ad (P_1)$ to $\Lambda^2\T N\simeq\so(\T N)$ which clearly verifies \eqref{epsi}.
\end{proof}

In the sequel, we will say that the parallel $\gg_1$-structure obtained in this way from an ideal $\gg_1$ of $\gg$ is a {\em reduction} of the initial parallel $\gg$-structure to the ideal $\gg_1$.

\subsection{Restriction to Riemannian factors} 

\begin{elem}\label{pg}
Assume that $(N,g^N)$ is the Riemannian product of $(N_1,g^{N_1})$ and $(N_2,g^{N_2})$. Then every parallel $\gg$-structure $(P,\gamma,\psi)$ on $(N,g^N)$ with the property that 
\be\label{imp}
\psi(\ad(P) )\subset\Lambda^2\T N_1\oplus \Lambda^2\T N_2\subset \Lambda^2\T N\ee 
induces parallel $\gg$-structures on the factors $(N_i,g^{N_i})$. 
\end{elem}
\begin{proof} Every point of $N_2$ defines an isometric embedding of $(N_1,g^{N_1})$ into $(N,g^N)$. By pull-back through this embedding one obtains a $G$-principal bundle $P_1$ over $N_1$ with connection $\gamma_1$. Moreover, the condition \eqref{imp} shows that $\psi$ defines by restriction a Lie algebra bundle morphism $\psi_1:\ad(P_1)\to \Lambda^2\T N_1$, which is clearly still parallel and satisfies \eqref{epsi}. The proof for $N_2$ is similar.
\end{proof}
Note that the condition \eqref{imp} is automatically satisfied for the factors of $N$ in the standard de Rham decomposition, see Lemmas \ref{le1} and \ref{le2} below.

\subsection{Whitney products} As a converse to Lemma \ref{g1} we have the following:
\begin{elem} \label{wp} Let $G_1$ and $G_2$ be compact Lie groups with Lie algebras $\gg_1$ and $\gg_2$ endowed with bi-invariant scalar products, and let $G:=G_1\times G_2$, with Lie algebra $\gg=\gg_1\oplus\gg_2$ endowed with the direct sum scalar product. If $(P_1,\gamma_1,\psi_1)$ and $(P_1,\gamma_1,\psi_1)$ are parallel $\gg_1$- and $\gg_2$-structures on $(N,g^N)$ such that $\psi_1(\ad(P_1))$ commutes with $\psi_2(\ad(P_2))$, then the Whitney product $(P_1\times P_2,\gamma_1+\gamma_2,\psi_1+\psi_2)$ is a parallel $\gg$-structure on $(N,g^N)$.
\end{elem}
\begin{proof} Everything is tautological, by noticing that the map
$$\psi_1+\psi_2:\ad(P_1\times P_2)=\ad(P_1)\oplus\ad(\P_2)\to\Lambda^2\T N$$ 
is a Lie algebra bundle morphism due to the commutation assumption.
\end{proof}

\begin{ede}\label{ndg} A parallel $\gg$-structure is called {\em non-degenerate} if for every orthogonal and parallel decomposition $\T N=D_1\oplus D_2$ and orthogonal decomposition $\gg=\gg_1\oplus \gg_2$ with $\gg_i$ Lie sub-algebras of $\gg$ satisfying $\psi(u\xi_1)\in\Lambda^2D_1$ and $\psi(u\xi_2)\in\Lambda^2D_2$ for every $u\in P$, $\xi_1\in\gg_1$ and $\xi_2\in\gg_2$, then either $D_1=0$ and $\gg_1=0$, or $D_2=0$ and $\gg_2=0$.\end{ede}

Equivalently, a parallel $\gg$-structure is non-degenerate if it is not locally a product of parallel $\gg$-structures, as described in §\ref{product}.

\begin{ere}\label{injective}
The morphism $\psi$ of a non-degenerate parallel $\gg$-structure is injective. Indeed, let $\gg_1\subset \gg$ be defined by $P\times_\Ad\gg_1=\ker\psi$, $\gg_2:=\gg_1^\perp$, $D_1=0$ and $D_2=\T N$. The conditions in Definition \ref{ndg} are clearly satisfied, and since $D_2\ne 0$, we necessarily have $\gg_1=0$.
\end{ere}

\begin{epr}\label{p76}
The (locally defined) base $N$ of the standard submersion of an indecomposable geometry with torsion of special type $(M,g,\tau)$ carries a non-degenerate parallel $\gg$-structure. Conversely, every non-degenerate parallel $\gg$-structure $(g^N,P,\gg,\gamma,\la\cdot,\cdot\ra,\psi)$ on a manifold $N$ induces an indecomposable geometry with torsion of special type on the total space $P$.
\end{epr}
\begin{proof}
By Proposition \ref{p75}, we only have to check that the indecomposability of $(M,g,\tau)$ is equivalent to the non-degeneracy of $(g^N,P,\gg,\gamma,\psi)$. Assume first that $(M,g,\tau)$ is indecomposable, and suppose that there is an orthogonal and parallel decomposition $\T N=D_1\oplus D_2$ and an orthogonal decomposition $\gg=\gg_1\oplus \gg_2$ with $\gg_i$ Lie sub-algebras of $\gg$ satisfying $\psi(u\xi_1)\in\Lambda^2D_1$ and $\psi(u\xi_2)\in\Lambda^2D_2$ for every $u\in P$, $\xi_1\in\gg_1$ and $\xi_2\in\gg_2$. Consider for $i=1,2$ the distributions $T_i$ on $M=P$ spanned by the horizontal lift of $D_i$ and by fundamental vector fields $\xi^*$ with $\xi\in \gg_i$. By \eqref{ta1}--\eqref{ta3}, if $T_1$ and $T_2$ are non-trivial, $(M,g,\tau)$ would be decomposable, according to Definition \ref{ri}. On the other hand $(M,g,\tau)$ is assumed to be indecomposable by Definition \ref{sgwt}, so we necessarily have $T_1=0$ or $T_2=0$. This shows that the $\gg$-structure of $N$ is non-degenerate.

Conversely, suppose that the parallel $\gg$-structure of $N$ is non-degenerate, and assume that $\T P=T_1\oplus T_2$ is an orthogonal $\nabla^\tau$-parallel decomposition of the tangent bundle of $P$, such that $\tau\in\Lambda^3 T_1\oplus \Lambda^3 T_2$. We denote by $\H:=\ker(\gamma)$ the horizontal distribution of $P$, and by $\V$ the vertical distribution (tangent to the fibers). Then $\T P=\H\oplus\V$ is another orthogonal $\nabla^\tau$-parallel decomposition, and from \eqref{tau} we have $\tau(\H,\H)\subset \V$, $\tau(\V,\V)\subset \V$ and $\tau(\V,\H)\subset \H$. Let $\V'$ be the set of vectors $V\in\V$ such $\tau_V|_\H$ is non-degenerate. For $V\in\V'$ we decompose $V=Y_1+Y_2$ with $Y_1\in T_1$ and $Y_2\in T_2$, and then we decompose $Y_i=X_i+V_i$ with $X_i\in \H$ and $V_i\in \V$ for $i=1,2$. We thus have $V=V_1+V_2$ and $X_1+X_2=0$.
Since $\tau\in\Lambda^3 T_1\oplus \Lambda^3 T_2$ we obtain for every $X\in H$:
\bea 0&=&\tau(Y_1,Y_2,X)=\tau(V_1+X_1,V_2-X_1,X)=\tau(V_1,V_2,X)+\tau(X_1,V_2,X)-\tau(V_1,X_1,X)\\&=&-\tau(V_1+V_2,X_1,X)=-\tau_V(X_1,X).\eea
The assumption that $\tau_V|_\H$ is non-degenerate thus shows that $X_1=0$, whence $V=V_1+V_2$. This shows that $(T_1\cap \V)\oplus (T_2\cap \V)$ contains the set $\V'$, which is dense in $\V$, so $$(T_1\cap \V)\oplus (T_2\cap \V)=\V.$$
The orthogonal complement of $T_i\cap \V$ in $T_i$ is clearly contained in $T_i\cap \H$, so a dimension count immediately shows that 
$$(T_1\cap \H)\oplus (T_2\cap \H)=\H.$$
Denoting $T_i\cap \V=:\V_i$ and $T_i\cap \H=:\H_i$, we thus get orthogonal $\nabla^\tau$-parallel decompositions $\V=\V_1\oplus\V_2$ and $\H=\H_1\oplus\H_2$. The assumption $\tau\in\Lambda^3 T_1\oplus \Lambda^3 T_2$ implies $\tau(\V_1,\V_2)=0$, whence $\tau(\V_i,\V_i)\subset\V_i$ for $i=1,2$. By \eqref{ta3}, this means that $\V_i$ are involutive distributions. Recalling that $\V=P\times_\Ad \gg$, this shows that the Lie algebra $\gg$ decomposes in an orthogonal direct sum of Lie sub-algebras $\gg=\gg_1\oplus\gg_2$ such that $\V_i=P\times_\Ad \gg_i$. 

Being $\nabla^\tau$-parallel, the distributions $\H_i$ project to parallel distributions $D_i$ on $N$. Indeed, if $V$ is a section of $\V$ and $X_1$ is a section of $\H_1$, then the horizontal part of $[V,X_1]$ reads
$$[V,X_1]_\H=(\nabla_VX_1-\nabla_{X_1}V)_\H=(\nabla^\tau_VX_1-\nabla^\tau_{X_1}V+2\tau(V,X_1))_\H=(\nabla^\tau_VX_1+2\tau(V,X_1))_\H\in\H_1.$$

The fact that $\tau(V_1,X_2)=0$ for every $V_1\in \V_1$ and $X_2\in\H_2$ is equivalent by \eqref{ta2} with the condition $\psi(\V_1)\subset\Lambda^2D_1$. Similarly, $\psi(\V_2)\subset\Lambda^2D_2$. From Definition \ref{ndg} we thus obtain that $\V_i=\H_i=0$ for some $i\in\{1,2\}$. This shows that $T_i=0$, so $(g,\tau)$ is indecomposable.
\end{proof}

\begin{exe}\label{e78} (i) If $(g^N,P,\gamma,\la\cdot,\cdot\ra,\psi)$ is a parallel $\u(1)$-structure on $N$, $\ad(P)$ has a global parallel section whose image by $\psi$ is a parallel 2-form on $N$. Conversely, if $\Omega$ is a parallel 2-form on $(N,g^N)$, let $\eta\in\Omega^1(N)$ be any locally defined primitive of $\Omega$. We define $P:=N\times \RM$ (viewed as a principal $\RM$-bundle) and $\gamma:=\d t+\eta\in\Omega^1(P)$ (everything is locally defined, and we omit writing down the pull-back signs, in order to keep notation simple). The adjoint bundle $\ad(P)$ is trivial, generated by a section called $1$. We define the parallel morphism $\psi:\ad(P)\to \Lambda^2 \T N$ by $\psi(1):=-\Omega$.
Then $(P,\gamma,\psi)$ is a non-degenerate $\u(1)$-structure on $(N,g^N)$. Indeed, $R^\gamma$ is equal to $\d\gamma=\Omega$, so it is parallel as a section in the trivial bundle $\ad(P)$, and the map $\psi$ satisfies \eqref{epsi} and is clearly a Lie algebra bundle morphism since the fibers of $\ad (P)$ are 1-dimensional. The non-degeneracy of the $\u(1)$-structure is clearly equivalent to the non-degeneracy of the corresponding 2-form. A non-degenerate parallel $\u(1)$-structure thus defines a Kähler structure on $N$, which is moreover unique up to sign when $(N,g^N)$ is irreducible.

(ii) More generally, if $(g^N,P,\gamma,\la\cdot,\cdot\ra,\psi)$ is a parallel $\u(1)^m$-structure on $N$, $\ad(P)$ is spanned by $m$ parallel sections, whose images by $\psi$ are $m$ parallel $2$-forms on $N$ whose associated endomorphisms mutually commute. Moreover these endomorphisms have no common kernel if the $\u(1)^m$-structure is non-degenerate. By diagonalizing them simultaneously  applying de Rham's decomposition theorem, we see that $(N,g^N)$ is a Riemannian product of Kähler manifolds. Conversely,  let  $N=N_1\times\ldots\times N_s$ be a Riemannian product of Kähler manifolds with fundamental 2-forms $\Omega_\alpha$, and let $c_{i\alpha}$ be real numbers for $\alpha\in\{1,\ldots,s\}$ and $i\in\{1,\ldots,m\}$. We consider the parallel forms $F_i:=\sum_\alpha c_{i\alpha}\Omega_\alpha$ on $N$ and some locally defined primitive $\eta_i\in\Omega^1(N)$ of $\Omega_i$ and denote by $P:=N\times \RM^m$, and $\gamma_i:=\d t_i+\eta_i\in\Omega^1(P)$. Then $\gamma_i$ are the components of a connection form on $P$ whose curvature form has components $F_i$. For every scalar product on $\u(1)^m$ one obtains a parallel $\u(1)^m$-structure on $N$ by choosing a parallel orthonormal basis $\xi_1,\ldots,\xi_m$  of $\ad(P)$ and defining $\psi(\xi_j):=-F_j$, so that \eqref{epsi} holds. 
\end{exe}

Once we have fixed Kähler structures on the factors $N_\alpha$, a parallel $\u(1)^m$-structure on $N=N_1\times\ldots\times N_s$ is thus determined by the $m\times s$ real matrix $\{c_{i\alpha}\}$. It is possible to express the non-degeneracy condition in terms of this matrix:

\begin{elem}\label{cdec} The above defined parallel $\u(1)^m$-structure is degenerate if and only if there exists a partition $\{1,\ldots,s\}=A\sqcup B$ and an orthogonal decomposition $\RM^m=V_1\oplus V_2$ with $V_1,V_2\ne 0$, such that $(c_{1\alpha},\ldots,c_{m\alpha})\in V_1$ for every $\alpha\in A$ and $(c_{1\beta},\ldots,c_{m\beta})\in V_2$ for every $\beta\in B$.
\end{elem}
\begin{proof} Consider $A,B,V_1,V_2$ satisfying the above condition. We define 
$$D_1:=\oplus_{\alpha\in A}\T N_\alpha, \qquad D_2:=\oplus_{\beta\in B}\T N_\beta,\qquad \gg_j:=\{\sum_{i=1}^mx_i\xi_i\ |\ x\in V_j\},\ j=1,2\ .$$
For every $x\in V_1$ we have 
$$\psi(\sum_i x_i\xi_i)=\sum_ix_i F_i=\sum_{i=1}^m\sum_{\alpha=1}^sx_ic_{i\alpha}\Omega_\alpha=\sum_{\alpha\in A}\sum_{i=1}^mx_ic_{i\alpha}\Omega_\alpha\ ,$$
showing that $\psi(\gg_1)$ vanishes on $D_2$, and similarly $\psi(\gg_2)$ vanishes on $D_1$. Since $\gg_1$ and $\gg_2$ are non vanishing, Lemma \ref{ndg} shows that the $\u(1)^m$-structure is degenerate. The converse statement can be proved similarly.
\end{proof}

\begin{ecor} If a parallel $\u(1)^m$-structure on $N=N_1\times\ldots\times N_s$ constructed as before is non-degenerate, then $m\le s$.
\end{ecor}
\begin{proof} Indeed, if $m>s$ then the $s$ vectors $(c_{1\alpha},\ldots,c_{m\alpha})$ span a strict subspace $V_1$ of $\RM^n$ (if $V_1=0$ we replace it by any proper subspace of $\RM^n$). Then $V_2:=V_1^\perp$, $A:=\{1,\ldots,s\}$ and $B:=\emptyset$ satisfy the condition of Lemma \ref{cdec}, so the structure is degenerate.
\end{proof}

\begin{exe}\label{e79}
If $(N,g^N)$ carries a parallel $\sp(1)$-structure, then by Lemma \ref{iredg} below, the structure is degenerate unless $(N,g^N)$ is de Rham irreducible, in which case it is quaternion-Kähler with positive scalar curvature by Proposition \ref{irre} below. 
Conversely, every quaternion-Kähler manifold with positive scalar curvature carries a parallel non-degenerate $\sp(1)$-struc\-ture defined by its Konishi bundle \cite{konishi} (see the appendix for the explicit calculations).
\end{exe}

\begin{exe}\label{ex1} Every symmetric space of compact type $N=L/G$ carries a natural parallel $\gg$-structure. Indeed, consider the natural metrics on $N$ and $\gg$ induced by an $\Ad_L$-invariant scalar product on the Lie algebra $\ll$ of $L$, and define $P:=L$, seen as $G$-principal bundle over $N$, with the connection $\gamma$ induced from the Levi-Civita connection of $N$. If $\mm$ denotes the orthogonal complement of $\gg$ in $\ll$, $\gamma$ is just the $\mm$-component of the Maurer-Cartan form of $L$.
Then the $G$-equivariant map $\phi:\gg\to\so(\mm)\simeq\Lambda^2\mm$, $a\mapsto \ad_a|_\mm$ induces a parallel Lie algebra bundle morphism
$\psi:\ad (P)\to\Lambda^2 \T N$ by $\psi(u\xi):=u\phi(\xi)$ for every $u\in P$ and $\xi\in\gg$. In order to check  \eqref{epsi}, let $u$ be a local section of $P$, $x,y\in \mm$ and $X:=ux$, $Y:=uy$ the corresponding local vector fields on $N$. Then for every $\xi\in\gg$ we have
\bea\langle u\xi,R^\gamma_{X,Y}\rangle&=&-\langle u\xi,u[x,y]\rangle=-\langle \xi,[x,y]\rangle=\langle [x,\xi],y\rangle=-\langle \ad_\xi( x),y\rangle\\
&=&-\langle \phi(\xi),x\wedge y\rangle=-g^N( u\phi(\xi),ux\wedge uy)=-g^N(\psi(u\xi),X\wedge Y)\ .\eea
More generally, according to Lemma \ref{g1}, a symmetric space of compact type  $N=L/H$ carries a canonical parallel $\gg$-structure for every ideal $\gg$ of the isotropy Lie algebra $\hh$. 
\end{exe}

Conversely, we have the following:
\begin{elem}\label{symg} Let $\gg$ be a semi-simple Lie algebra of compact type and let  $(P,\gamma,\la\cdot,\cdot\ra,\psi)$ be a parallel $\gg$-structure on a locally symmetric space $(N=L/H, g^N)$ of compact type with $\psi$ fiberwise injective. Then $\gg$ is an ideal of $\hh$ and the $\gg$-structure on $N$  obtained as in Example \ref{ex1} by reduction of the canonical parallel $\hh$-structure on $N$ to $\gg$.
\end{elem}
\begin{proof} Let us consider as usual the curvature tensors $R^\gamma$ of $(P,\gamma)$ and $R^N$ of $(N,g^N)$ as bundle morphisms $R^\gamma:\Lambda^2\T N\to\ad(P)$ and $R^N:\Lambda^2\T N\to \Lambda^2\T N$ by 
$$R^\gamma(\omega):=\tfrac12\sum_{i,j}\omega(e_i,e_j)R^\gamma_{e_i,e_j},\qquad R^N(\omega)(X,Y):=\tfrac12\sum_{i,j}\omega(e_i,e_j)g^N(R^N_{e_i,e_j}X,Y)\ .$$
Since $N=L/H$ is of compact type, the metric on $N$ is defined by a bi-invariant scalar product on the Lie algebra $\ll$ of $L$. Let $\hh$ denote the Lie algebra of $H$ and let $\mm$ be its orthogonal complement in $\ll$. The isotropy representation of $\hh$ on $\mm$ defines an embedding of $\hh$ in $\Lambda^2\mm$, and we denote by $\hh^\perp$ its orthogonal complement, so that $\Lambda^2 \mm=\hh\oplus\hh^\perp$. Correspondingly, the bundle $\Lambda^2 \T N$ decomposes in an orthogonal direct sum $\Lambda^2 \T N=\hh N\oplus \hh^\perp N$. As an endomorphism of $\Lambda^2\T N$, the Riemannian curvature tensor $R^N$ takes values in $\hh N$, so by pair symmetry $R^N$ vanishes on $\hh^\perp N$.

We now use \eqref{rgrn}, which in the present context reads 
\be\label{pr}\psi(R^\gamma(\omega)s)=[R^N(\omega),\psi(s)]\ee
for every $s\in\ad(P)$ and $\omega\in\Lambda^2\T N$. Applying this to some $\omega\in \hh^\perp N$ and using the vanishing of $R^N$ on $\hh^\perp N$, together with the injectivity of $\psi$ yields $R^\gamma(\omega)s=0$ for every $\omega\in \hh^\perp N$ and $s\in\ad(P)$. Moreover, $R^\gamma(\omega)s=[R^\gamma(\omega),s]$ and since $\gg$ is semi-simple, this shows that $R^\gamma(\omega)=0$ for every $\omega\in \hh^\perp N$. From \eqref{epsi} we thus get that $\psi(\ad(P))$ is orthogonal to $\hh^\perp N$, i.e. $\psi(\ad(P))\subset \hh N$. Since $\psi$ is a Lie algebra bundle morphism, this shows that $\gg$ is identified with a Lie sub-algebra of $\hh$. Moreover, it is well known that $R^N$ is an isomorphism of $\hh N$, so \eqref{pr} shows that $\gg$ is actually an ideal of $\hh$. \end{proof}

\section{Classification of non-degenerate parallel $\gg$-structures}\label{se8}

The aim of this section is the following classification result:

\begin{ath}\label{pgs}
Let $\gg$ be a Lie algebra of compact type and let $(g^N,P,\gg,\gamma,\la\cdot,\cdot\ra,\psi)$ be a non-degenerate parallel $\gg$-structure on a manifold $N$. Then either 
\begin{itemize}
\item $N$ is quaternion-Kähler with positive scalar curvature, $\gg=\sp(1)$ and $P$ is the Konishi bundle like in Example \ref{e79}, or 
\item $N=L/H$ is an irreducible locally symmetric space of compact type, $\gg$ is isomorphic to a semi-simple factor of $\hh$ and the parallel $\gg$-structure is the one described in Example \ref{ex1}, or
\item $N$ is locally a Riemannian product $N=N_1\times \ldots\times N_p\times S_1\times\ldots\times S_q$ with $N_\alpha$ Kähler for $\alpha\in\{1,\ldots,p\}$, $S_\beta=L_\beta/\U(1)H_\beta$ Hermitian symmetric of compact type for $\beta\in\{1,\ldots,q\}$, $\gg=\u(1)^m\oplus \k\oplus\ldots\oplus\k_q$ and $\k_\beta$ a non-zero factor of $\hh_\beta$. The parallel $\gg$-structure on $N$ is the Whitney product of a parallel $\u(1)^m$-structure on $N$ like in Example \ref{e78} (ii) and a parallel $\k\oplus\ldots\oplus\k_q$-structure on $N$ which is the Riemannian product of the canonical parallel $\k_\beta$-structures on $S_\beta$ (Example \ref{ex1}) and the $0$-structures on the factors $N_\alpha$.
\end{itemize}
\end{ath}

\begin{proof} 
Let $(N,g^N)=N_0\times N_1\times \ldots\times N_s$ be the local de Rham decomposition of $N$, with $N_0$ flat and $N_i$ irreducible for $i\ge 1$. 
We decompose the Lie algebra $\gg$ as $\gg=\z\oplus \gg_1\oplus\ldots\oplus\gg_l$, where $\z$ denotes its center and $\gg_i$ are simple Lie algebras of compact type. Since the scalar product is $\ad_\gg$-invariant, this decomposition can be chosen to be orthogonal. We define the corresponding Lie algebra bundles $\z P:=P\times _\ad \z$ (which is actually trivial) and $\gg_i P:= P\times_\ad\gg_i$, so that $\ad (P)=\z P\oplus \gg_1 P\oplus\ldots\oplus\gg_l P$.
Recall that by Remark \ref{injective} the map $\psi$ is injective since the parallel $\gg$-structure is assumed to be
non-degenerate.

\begin{elem}\label{le1} If the Lie algebra bundle morphism $\psi$ is injective, then 
$$\psi(\gg_i P)\subset \bigoplus_{\alpha\ge 1} \Lambda^2 \T N_\alpha,\qquad\forall\ i\in\{1,\ldots,l\}\ .$$
\end{elem}
\begin{proof} Let $\alpha,\beta\in\{0,\ldots,s\}$ be either different or both equal to $0$, and let $X_\alpha$, $X_\beta$ be tangent vectors to $N_\alpha$ and $N_\beta$ respectively. From the symmetries of the Riemannian curvature tensor we obtain $R^N_{X_\alpha,X_\beta}=0$. Using \eqref{rgrn} we get for every $\xi\in\gg$
$$0=R^N_{X_\alpha,X_\beta}(\psi(u\xi))=\psi(R^\gamma_{X_\alpha,X_\beta}(u\xi)),$$
whence $R^\gamma_{X_\alpha,X_\beta}(u\xi)=0$ by the injectivity assumption. On the other hand 
$$0=R^\gamma_{X_\alpha,X_\beta}(u\xi)=[R^\gamma_{X_\alpha,X_\beta},u\xi]$$ 
as local sections of $\ad (P)$, so $R^\gamma_{X_\alpha,X_\beta}$ is a section of $\z P$. This shows that
 for every $i\ge 1$ and $\xi\in\gg_i$ we have
 $$0=\la R^\gamma_{X_\alpha,X_\beta},u\xi\ra=-g^N(\psi(u\xi),X_\alpha\wedge X_\beta),$$
 so finally $\psi(u\xi)$ is orthogonal to the sub-bundles $\Lambda^2 \T N_0$ and to $\T N_\alpha\wedge \T N_\beta$ of $\Lambda^2\T N$ for all $\alpha\neq \beta$.
\end{proof}

\begin{elem}\label{le2}
$\psi(\z P) \subset \oplus_{\alpha \ge 0} \Lambda^2 \T N_\alpha$
\end{elem}
\begin{proof}
Let $\xi_1,\ldots,\xi_m$ be an orthonormal basis of the center $\z$, inducing 
a global orthonormal parallel basis $\hat \xi_1, \ldots, \hat  \xi_m$ of $\z P$. Let $\psi (\hat \xi_i) =: F_i$ be the corresponding
parallel skew-symmetric endomorphisms of $\T N$. 

The de Rham theorem shows that the restricted holonomy group of $N$ is isomorphic to a product $K_0\times \ldots\times K_s$, and $\T N$ is associated to a representation of this group on a direct sum $\hh_0\oplus\ldots\oplus \hh_s$, (with $\hh_\alpha$ corresponding to $\T N_\alpha$), such that $K_0=0$ and for every $\alpha\ge 1$, $K_\alpha$ acts irreducibly on $\hh_\alpha$ and trivially on $\hh_\beta$ for $\beta\ne \alpha$. Each parallel endomorphism $F_i$ corresponds to an equivariant map of this representation. On the other hand, every equivariant map clearly preserves each summand $\hh_\alpha$, thus showing that $F_i(\T N_\alpha) \subset \T N_\alpha$ for all $\alpha \ge 0$.
\end{proof}

We now define $I=\{1,\ldots,l+m\}$ and denote by $\gg_{l+i}$ the sub-algebra generated by $\xi_i$ for $1\le i\le m$.
Let $\gg = \oplus_{i \in I}\gg_i$ be the decomposition of $\gg$ with corresponding $\nabla^\gamma$-parallel 
decomposition of $\ad(P) = \oplus_{i\in I} \gg_i P$ of the adjoint bundle. 
By Lemmas \ref{le1} and \ref{le2}, the parallel bundle morphism
$\psi$ maps $\ad (P)$ to $\oplus_{\alpha \in A} \Lambda^2 \T N_\alpha$, where $A=\{0,\ldots,s\}$. Let $\pi_\alpha$ be the projection of $\Lambda^2\T N$ onto the sub-bundle $\Lambda^2\T N_\alpha$. We use the notation 
\be\label{eai}E_{\alpha i} := \left. \pi_\alpha (\psi(\gg_i P))\right|_{N_\alpha}
\ee 
for the corresponding parallel sub-bundle of 
$\Lambda^2 \T N_\alpha\to N_\alpha$. In other words, $E_{\alpha i}$ is the parallel sub-bundle of $\Lambda^2\T N_\alpha$ corresponding to parallel $\gg_i$-structure on $N_\alpha$ obtained by reducing the initial $\gg$-structure to $\gg_i$ (Lemma \ref{g1}) and then restricting to $N_\alpha$ (Lemma \ref{pg}).

The next result can be seen as a generalization to parallel $\gg$-structures of the well known fact that quaternion-Kähler manifolds are irreducible.

\begin{elem}\label{iredg} If $E_{\alpha i}\ne 0$ for some $i\in \{1,\ldots,l\}$, then $E_{\beta i}=0$ for every $\beta\in \{0,\ldots,s\}\setminus\alpha$.
\end{elem}
\begin{proof} The first Bianchi identity and the Riemannian curvature identities show that for every $\alpha\ne \beta$ and tangent vectors $X,Y\in \T N_\alpha$ and $Z\in \T N_\beta$ one has 
\beq\label{rv} R^N_{X,Y}Z=0.\eeq

Let $x$ be any point of $N$ and $u$ an element in the fibre of $P$ over $x$. The hypothesis gives the existence of two vectors $X,Y\in \T_xN_\alpha$ and some $\xi\in\gg_i$ such that $g^N(\psi(u\xi),X\wedge Y)\ne 0$. By \eqref{epsi} we thus get $\la u\xi,R^\gamma_{X,Y}\ra\ne 0$. Since $\gg_i$ is a simple Lie algebra, it has no center, so there exists some $\zeta\in \gg_i$ with $[R^\gamma_{X,Y},u\zeta]\ne 0$. For every $\beta\in \A$, the map 
$$\pi_\beta\circ\psi:\ad (P)\to \Lambda^2\T N_\beta$$
is a parallel Lie algebra bundle morphism, so using again the fact that $\gg_i$ is simple, $\pi_\beta\circ\psi$ either vanishes identically, or is injective. Assume for a contradiction that $E_{\beta i}\ne 0$ for some $\beta\in \{0,\ldots,s\}\setminus\alpha$. Then from \eqref{rgrn} we obtain
$$0\ne \pi_\beta\circ\psi([R^\gamma_{X,Y},u\zeta])=\pi_\beta([\psi(R^\gamma_{X,Y},\psi(u\zeta)])=\pi_\beta([R^N_{X,Y},\psi(u\zeta)]).$$
This shows that there exists $Z\in \T_x N_\beta$ such that 
$$[R^N_{X,Y},\psi(u\zeta)](Z)\ne 0.$$
Denoting by $F:=\psi(u\zeta)$, this reads $R^N_{X,Y}FZ\ne FR^N_{X,Y}Z$. On the other hand, $FZ\in \T_x N_\beta$ by Lemma \ref{le1}, so both $R^N_{X,Y}FZ$ and $FR^N_{X,Y}Z$ vanish from \eqref{rv}. This contradiction concludes the proof.
\end{proof}

\begin{elem}\label{part}
If there are partitions $A = A_1 \sqcup A_2$ and $I = I_1 \sqcup I_2$  of the two index sets
$I$ and $A$ such that $E_{\alpha i} =0$ for all $\alpha \in A_2, i \in I_1$ and for all $\alpha \in A_1, i \in I_2$, then either $A_1=I_1=\emptyset$ or $A_2=I_2=\emptyset$.
\end{elem}
\begin{proof} The argument is similar to the one used in the proof of Proposition \ref{p75}. Consider such partitions  $A = A_1 \sqcup A_2$ and $I = I_1 \sqcup I_2$. For $i=1, 2$ we define the distributions $T_i$ on $M:=P$ spanned by the horizontal lifts of vectors in $\bigoplus_{\alpha\in A_i}\T N_\alpha$ and by fundamental vector fields $\xi^*$ with $\xi\in \bigoplus_{j\in I_i}\gg_j$. We claim that $\tau\in\Lambda^3 T_1\oplus \Lambda^3 T_2$.

It is enough to show that $\tau(U,V,W)=0$ whenever two of the vectors $U,V,W$ belong to $T_1$ and one to $T_2$, and by multi-linearity, one can assume each of them is  either a horizontal lift or a vertical fundamental vector field. Using Lemmas \ref{tors} and \ref{l73}, we are left with two cases:

a) $U,V,W$ are all vertical, and $U=\xi^*,V=\zeta^*\in T_1$, and $W=\eta^*\in T_2$. Then by \eqref{tau} we have 
$$\tau(U,V,W)=-\tfrac12\la[\xi,\zeta],\eta\ra=0,$$
since $\xi$ and $\zeta$ belong to the sub-algebra $\bigoplus_{j\in I_1}\gg_j$ of $\gg$, which is orthogonal to $\bigoplus_{j\in I_2}\gg_j$, which contains $\eta$.

b) $U=\tilde X$ and $V=\tilde Y$ are horizontal lifts and $W=\xi^*$, with $X\in \T N_\alpha$, $Y\in\T N_\beta$, $\xi\in \gg_i$, and either $\alpha,\beta\in A_1$, $i\in I_2$, or $\alpha\in A_1$, $\beta\in A_2$, $i\in I_2$. In both cases we have by \eqref{tau} 
$$\tau(U,V,W)=\tfrac12\la\Omega^\gamma(\tilde X,\tilde Y),\xi\ra=\tfrac12\la u^{-1}R^\gamma_{X,Y},\xi\ra=-\tfrac12 g^N(\psi(u\xi),X\wedge Y).$$
If $\alpha=\beta\in A_1$, $i\in I_2$, this expression vanishes by the assumption that $E_{\alpha i} := \left. \pi_\alpha (\psi(\gg_i P))\right|_{N_\alpha}$ vanishes.
If $\alpha\ne \beta$, this expression vanishes by Lemmas \ref{le1} and \ref{le2}.

If $T_1$ and $T_2$ are non-vanishing, the decomposition $\T M=T_1\oplus T_2$ satisfies the decomposability conditions in Definition \ref{ri}. On the other hand, $(M,g,\tau)$ is assumed to be indecomposable by Definition \ref{sgwt}, we necessarily have $T_1=0$ or $T_2=0$, thus proving that  either $A_1=I_1=\emptyset$ or $A_2=I_2=\emptyset$.
\end{proof}

We will now restrict our attention to the case where $N$ is irreducible.

\begin{epr}\label{irre}
Let $(N,g^N)$ be an irreducible Riemannian manifold with a parallel $\gg$-structure such that the morphism $\psi$ is not identically zero. Then one of the following three cases may occur:
\begin{itemize}
\item The Lie algebra sub-bundle $\mathrm{Im}(\psi)\simeq \ad (P)/\Ker (\psi)$ of $\Lambda^2 \T N$ is a line bundle, and $N$ is Kähler;
\item Each fiber of $\mathrm{Im}(\psi)$ is isomorphic to $\sp(1)$ and $N$ is quaternion-Kähler with positive scalar curvature;
\item $N$ is locally symmetric of compact type.
\end{itemize}
\end{epr}
\begin{proof} The sub-bundle $VN:=\mathrm{Im}(\psi)$ of $\Lambda^2 \T N\simeq \End^-(\T N)$ is a parallel sub-bundle, closed under the usual bracket of endomorphisms. It corresponds to an invariant subspace $V$ of $\Lambda^2 T$, where $T$ denotes the holonomy representation of $N$.

Consider the Riemannian curvature tensor $R^N$ of $N$, also viewed as an endomorphism $R^N:\Lambda^2 \T N\to \Lambda^2 \T N$ by 
\beq\label{nms}g^N(R^N(X\wedge Y),Z\wedge W)=g^N(R^N_{X,Y}Z,W)\ .\eeq

Using the relation between the curvatures of $\ad(P)$ and $N$ obtained in \eqref{e50}, we see that the endomorphism $R^\perp$ of $\Lambda^2\T N$ defined by
\beq\label{rper}R^\perp:=R^N-\psi\circ R^\gamma\eeq 
takes values in the centralizer of $VN$.

Let us decompose $V$ in an orthogonal direct sum of irreducible components $V=V_1\oplus\ldots\oplus V_k$ and correspondingly $VN=\oplus_a V_aN$. 

Schur's lemma shows that there exist positive real numbers $\l_a$ such that $\psi\circ\psi^*=\sum_a\lambda_a \pi_a$, where $\pi_a$ denotes the orthogonal projection from $\Lambda^2\T N$ to $V_a N$.  On the other hand, \eqref{epsi} shows that $\psi\circ R^\gamma=-\psi\circ\psi^*$, so by \eqref{rper} we obtain 
\beq\label{rrr}R^N=R^\perp-\sum_a\l_a \pi_a\ ,\eeq
where we recall that $R^\perp$ takes values in the centralizer of $VN$. We introduce the symmetric endomorphisms of $\T N$
$$S_a(X):=\sum_ie_i\lrcorner \pi_a(e_i\wedge X), \qquad S^\perp(X):=\sum_ie_i\lrcorner R^\perp(e_i\wedge X),$$
for every local orthonormal basis $\{e_i\}$ of $\T N$. 

We fix some $a\in\{1,\ldots,k\}$ and consider any orthonormal basis $\{A_s\}$ of $V_a N$ (with respect to the natural scalar product on 2-forms induced by $g^N$). The endomorphism $\sum_s A_s^2$ of $\T N$ is clearly parallel, so by the irreducibility of $N$, there exists some positive constant $b_a$ such that $\sum_s A_s^2=-b_a\Id$. We thus have for every $X\in \T N$:
\beq\label{sa}S_a(X)=\sum_{i,s}g^N(e_i\wedge X,A_s)A_s(e_i)=-\sum_sA_s^2(X)=b_a X.\eeq

 Moreover, since $V_a$ is a simple Lie algebra, the Casimir element of its adjoint representation is a multiple of the identity.
Consequently, there exists a positive constant $c_a$ such that $\sum_s[A_s,[A_s,A]]=-c_a A$ for every section $A$ of  $V_a N$. Consequently we have:
\beq\label{cs} -c_a A=\sum_s[A_s,[A_s,A]]=\sum_sA_s^2A+A\sum_s A_s^2-2\sum_s A_s A A_s,
\eeq
whence
\beq\label{rica}\sum_s A_s A A_s=(\tfrac12{c_a}-b_a) A,\qquad\forall A\in V_a N.\eeq

Let $A$ be any section of $V_a N$. Using \eqref{rrr}, \eqref{sa}, the first Bianchi identity for $R^N$ and the fact that $A$ commutes with the images of $R^\perp$ and of $\pi_b$ for every $b\ne a$, we obtain for every tangent vector $X$:
\bea R^\perp(A)(X)&=&R^N(A)(X)+\lambda_a A(X)=\tfrac12\sum_i R^N_{e_i,Ae_i}X+\lambda_a A(X)\\
&=&-\tfrac12\sum_i R^N_{X,e_i}Ae_i-\tfrac12\sum_i R^N_{Ae_i,X}e_i+\lambda_a A(X)=\sum_i R^N_{e_i,X}Ae_i+\lambda_a A(X)\\
&=&\sum_i R^\perp(e_i\wedge X)(Ae_i)-\sum_{i,b}\l_b\pi_b(e_i\wedge X)(Ae_i)+\lambda_a A(X)\\
&=&AS^\perp(X)-\sum_{b\ne a}\l_b b_b A(X)-\l_a\sum_{i}\pi_a(e_i\wedge X)(Ae_i)+\lambda_a A(X).\eea
On the other hand, we have 
\bea\sum_{i}\pi_a(e_i\wedge X)(Ae_i)&=&\sum_{i,s}g^N(A_s,e_i\wedge X)A_sAe_i=-\sum_{i,s}g^N(A_s(X),e_i)A_sAe_i\\&=&-\sum_sA_sAA_s(X),\eea
so using \eqref{rica} we obtain 
\beq\label{rp}R^\perp(A)=AS^\perp+d_a A,\qquad\forall A\in V_a N,
\eeq
where
$$d_a:=-\sum_{b\ne a}\l_b b_b+\l_a(\tfrac12{c_a}-b_a)+\l_a=-\sum_{b}\l_b b_b+\l_a(1+\tfrac12{c_a}).$$

Since $A$ and $R^\perp(A)$ are skew-symmetric and $S^\perp$ is symmetric, \eqref{rp} shows that $S^\perp$ commutes with $A$ for every $A\in VN$. On the other hand, $R^\perp(A)$ commutes with every $B\in VN$, so using \eqref{rp} again, we obtain
\beq\label{ricv}(S^\perp+d_a\,\id)\circ [A,B]=0\eeq
for every $A,B\in VN$. Let us denote by $DN$ the parallel sub-bundle of $\T N$ spanned by the images of all endomorphisms of the form $[A,B]$ with  $A,B\in VN$. 

Then the irreducibility of $N$ implies that either $DN=0$ or $DN=\T N$.

{\bf Case 1:} $DN=0$. In this  case $VN$ is an Abelian Lie algebra sub-bundle of $\Lambda^2 \T N$. Thus $\psi$ vanishes on all sub-bundles $\gg_iP$ for $i\in\{1,\ldots,l\}$, so $VN$ is spanned by the parallel commuting endomorphisms $\psi(\hat\xi_1),\ldots,\psi(\hat\xi_m)$ defined in the proof of Lemma \ref{le2}. Using again the irreducibility of $N$ we obtain that $V$ has dimension 1, hence $VN$ is generated by a parallel endomorphism whose square is proportional to the identity, and thus $N$ is Kähler.

{\bf Case 2:} $DN=\T N$.
In this case, $S^\perp+d_a\,\id=0$ by \eqref{ricv}. In particular $d_a$ is independent of $a$, whence $\lambda_a(1+\frac{c_a}2)=:r>0$ for every $a$. By \eqref{rrr} we obtain
$$\Ric^N=\sum_a \l_a S_a-S^\perp=(\sum_a \l_a b_a)\id-(\sum_b\l_b b_b-r)\id=r\,\id,$$ 
and thus $N$ is Einstein and has positive scalar curvature.  If $N$ is locally symmetric, we are in the last case of the proposition. If $N$ is not locally symmetric, we examine the possible holonomy representations of $N$ given by the Berger-Simons holonomy theorem:

1. If $\Hol_0(N)$ is one of $\G_2\subset \SO(7)$, $\Spin(7)\subset \SO(8)$, $\SU(m)\subset\SO(2m)$ or $\Sp(q)\subset\SO(4q)$, then $N$ is Ricci-flat, so it cannot have positive scalar curvature.

2. If $\Hol_0(N)=\SO(n)$ and $T=\RM^n$, then $\Lambda^2 T$ is irreducible and has no center, unless $n=4$, when $\Lambda^2 T=\Lambda^+ T\oplus \Lambda^- T$. Up to a change of orientation for $N$ one can assume that $VN=\Lambda^+\T N$. Let us denote by $R^+$ the orthogonal projection from $\Lambda^2 \T N$ onto $\Lambda^+\T N$. 
From \eqref{rrr}, the curvature endomorphism of $N$ can be written $R^N= R^\perp-\l R^+$, where $\l>0$ and $R^\perp$ takes values in the centralizer of $VN$, which in the present situation is $\Lambda^- \T N$. 
Using the well known decomposition of the curvature operator in dimension 4 as
\be\label{r}R^N=-\begin{pmatrix} \frac{\mathrm{Scal}}{12}\Id+W^+& \frac12 \widetilde{\Ric^N_0}\\ & \\ \frac12 \widetilde{\Ric^N_0}& \frac{\mathrm{Scal}}{12}\Id+W^-\end{pmatrix},\ee
(where $\widetilde{\Ric^N_0}$ denotes the Kulkarni-Nomizu product of $g^N$ with the trace-less Ricci tensor of $N$) we thus obtain that $\mathrm{Scal}=12\lambda>0$, $W^+=0$ and $\Ric^N_0=0$, so $N$ is anti-selfdual and Einstein, which corresponds to the quaternion-Kähler condition in dimension 4.

3. If $\Hol_0(N)=\U(m)$ and $T=\RM^{2m}$, then the decomposition of $\Lambda^2 T$ in irreducible components reads $\so(2m)\simeq \Lambda^2 T=\u(1)\oplus \su(m)\oplus \mm$, where $\mm\simeq\Lambda^{(2,0)+(0,2)}T$ is isomorphic to the isotropy representation of the symmetric space $\SO(2m)/\U(m)$ (and thus verifies $[\mm,\mm]=\u(1)\oplus\su(m)$). Consequently, the only $\su(m)$-invariant Lie sub-algebras of $\so(2m)$ of dimension larger than 1 are $\su(m)$ and $\u(1)\oplus\su(m)$. Their centralizers in $\so(2m)$ are both equal to $\u(1)$. Geometrically, this means that $N$ is a Kähler manifold such that the endomorphism  $(\nabla^N_XR^N)_{Y,Z}$ is proportional to the complex structure for every tangent vectors $X,Y,Z$. This easily implies that $\nabla^NR^N=0$. Indeed, assume that $(\nabla^N_XR^N)_{Y,Z}=T(X,Y,Z)J$ for some tensor $T$. Using the second Bianchi identity we obtain for every tangent vectors $A,B,C,Y,Z$:
$$0=\mathop{\mathfrak{S}}_{A,B,C}g^N((\nabla^N_AR^N)_{Y,Z}B,C)=\mathop{\mathfrak{S}}_{A,B,C}T(A,Y,Z)g^N(JB,C).$$
Taking $B=JC$ of unit length and orthogonal to $A$ and $JA$ yields $T(A,Y,Z)=0$ for every $A,Y,Z$. Thus $T=0$, so $N$ is locally symmetric.

4. If $\Hol_0(N)=\Sp(q)\cdot\Sp(1)$ and $T=\RM^{4q}$ with $q\ge 2$, then $N$ is quaternion-Kähler. It remains to check that the fibers of $\psi(\gg P)$ are isomorphic to $\sp(1)$. It is well known that the decomposition of $\Lambda^2 T$ in irreducible summands is  $\Lambda^2 T=\sp(q)\oplus\sp(1)\oplus \mm$. We denote by  $\Lambda^2\T N=\sp(q)N\oplus\sp(1)N\oplus \mm N$ the corresponding decomposition of the bundle of 2-forms on $N$. We claim that $[\mm,\mm]$ contains $\sp(q)\oplus\sp(1)$. Indeed, if $[\mm,\mm]$ were orthogonal to some non-zero element in $\sp(q)\oplus\sp(1)$, then this element would commute with each element of $\mm$, and this would contradict the fact that the isotropy representation of the symmetric space $\SO(4q)/\U(2q)$ is faithful.

The centralizer of $\sp(q)\oplus\sp(1)$ in $\so(4q)$ clearly vanishes, so we are left with two possibilities: either $V=\sp(1)$ (in which case we are done), or $V=\sp(q)$. 

We will show that this last case is impossible. Indeed, if $V=\sp(q)$, \eqref{rrr} reads $R^N=\R^\perp-\l R_V$ for some positive constant $\lambda$, where $R_V$ denotes the projection on $VN$ and $R^\perp$ is a symmetric endomorphism of $\sp(1) N$ satisfying the first Bianchi identity. At every point of $N$ one can diagonalize $R^\perp$ in an orthonormal basis $\omega_1$, $\omega_2$, $\omega_3$ of $\sp(1)N$ so that 
$$R^\perp(X\wedge Y)=\tfrac1{2q}\sum_a\l_a g^N(X\wedge Y,\omega_a)\omega_a.$$ 
An easy computation then shows that the Bianchi condition $\sum_{i,j}e_i\wedge e_j\wedge R^\perp(e_i\wedge e_j)=0$ is equivalent to $\sum_a \l_a \omega_a\wedge\omega_a=0$. On the other hand, for $q\ge 2$ the 4-forms $\omega_a^2$ are linearly independent, so $R^\perp=0$, which shows that $R^N$ is parallel. Since $N$ was assumed to be non locally symmetric, this case is impossible, so the proposition is proved.
\end{proof}

By Lemma \ref{pg}, together with Lemmas \ref{le1} and \ref{le2}, we see that every factor $N_\alpha$ of $N$ (including the flat factor $N_0$) inherits a parallel $\gg$-structure $(P_\alpha,\gamma_\alpha,\psi_\alpha)$. We will distinguish two cases:

{\bf Case 1.} Assume first that there exists a factor $N_\alpha$ such that the sub-bundles $E_{\alpha j}$ defined in \eqref{eai} vanish for every $j\in \{l+1,\ldots,l+m\}$. 

We consider the partitions $A = A_1 \sqcup A_2$ and $I = I_1 \sqcup I_2$  of the two index sets $I=\{1,\ldots,l+m\}$ and $A=\{0,\ldots,s\}$ defined by 
$$A_1:=\{\alpha\},\qquad A_2:=A\setminus \{\alpha\},\qquad I_1=\{i\in I\ |\ E_{\alpha i}\ne0\},\qquad I_2=\{i\in I\ |\ E_{\alpha i}=0\}.$$ 
By Lemma \ref{iredg} we have that $E_{\beta i} =0$ for all $\beta \in A_2, i \in I_1$, and by the very definition of $I_2$ we have $E_{\beta i} =0$ for all $\beta \in A_1, i \in I_2$. Moreover $A_1$ is non-empty, so by Lemma \ref{part} we must have $A_2=I_2=\emptyset$. Thus $N=N_\alpha$ is irreducible.
By Proposition \ref{irre}, $N$ is either
a non locally symmetric quaternion-Kähler manifold with positive scalar curvature as in Example \ref{e78} (iii), or a locally symmetric space of compact type $L/H$. In the latter case, Lemma \ref{symg} shows that $\gg$ is an ideal of the Lie algebra $\hh$ of $H$ and the parallel $\gg$-structure on $N$ is the reduction of the canonical parallel $\hh$-structure of $L/H$ to $\gg$.

{\bf Case 2.} For every $\alpha\in A$, there exists $j\in \{l+1,\ldots,l+m\}$ such that $E_{\alpha j}\ne0$. By Lemmas \ref{g1} and \ref{pg}, the reduction of the $\gg$-structure to the element of the center of $\gg$ generated by $\xi_j$ defines a non-vanishing parallel 2-form on $N_\alpha$, so by irreducibility, $N_\alpha$ is Kähler for every $\alpha\in A\setminus \{0\}$. The same holds for $N_0$, except that here the Kähler structure is not unique (one might need to further decompose $N_0$ into a product of flat Kähler factors, but we don't want to insist on this). The important fact is that the reduction of the $\gg$-structure to the center $\z$ of $\gg$ is an Abelian $\gg$-structure on $N$, which can be written as in Example \ref{e78} (ii). 

We now denote by $$A':=\{\alpha\in A\ |\ E_{\alpha i}=0\ \forall i\in\{1,\ldots,l\}\},\qquad A'':=\{\alpha\in A\ |\ \exists  i\in\{1,\ldots,l\}, \ E_{\alpha i}\ne 0\}.$$ 
By Lemma \ref{le1}, $0\in A'$. By Proposition \ref{irre}, for each $\alpha\in A''$, the corresponding factor is locally symmetric, $N_\alpha =L_\alpha/H_\alpha$, so being Kähler, it is in fact Hermitian symmetric. By Lemma \ref{symg}, the reduction of the parallel $\gg$-structure on $N$ to the semi-simple part of $\gg$, followed by restriction to $N_\alpha$ is a reduction of the canonical parallel $\hh_\alpha$-structure of $L_\alpha/H_\alpha$ to a semi-simple factor of $\hh_\alpha$.

Finally, the parallel $\gg$-structure on $N$ is the Whitney product (Lemma \ref{wp}) of its reductions to $\z$ and to the semi-simple part of $\gg$, which is exactly the last case in the theorem.
\end{proof}

Note that Proposition \ref{p75} together with the above classification of Riemannian manifolds carrying non-degenerate parallel $\gg$-structures (Theorem \ref{pgs}), yield the classification of geometries with torsion of special type.

\section{Appendix. Geometries with parallel curvature over quaternion-Kähler manifolds}

In this appendix we will show, by explicit calculations, how the notions of parallel $\gg$-structures and geometries with parallel curvature (in the special case where $\gg=\sp(1)$ and the base is quaternion-Kähler) fit with the examples of geometries with parallel skew-symmetric torsion in the literature, with special emphasis on the so-called $3$-$(\alpha,\delta)$-Sasakian structures. 

Recall that a $3$-$(\alpha,\delta)$-Sasakian structure (\cite[Definition 2.2.1]{ad}) on a $4m+3$-dimensional manifold $M$ is an almost 3-contact metric structure $(g,\f_i,\xi_i,\eta_i)$ (see \cite[Definition 1.2.2]{ad}), such that for every even permutation $\{i,j,k\}$ of $\{1,2,3\}$:
\be\label{3ad} \d\eta_i=2\alpha\Phi^H_i-2\delta\eta_j\wedge\eta_k\ ,
\ee
where $\Phi^H_i:=\Phi_i+\eta_j\wedge\eta_k$ and $\Phi_i(X,Y):=g(X,\f_i(Y))$.

Let now $(N^{4m},g^N)$ be a quaternion-Kähler manifold (see \cite{besse} or \cite{ish}). By definition, there exists a 3-dimensional $\nabla^{g^N}$-parallel sub-bundle $F\subset\End^-(\T N)$ of skew-symmetric endomorphisms locally spanned by almost Hermitian structures $I,J,K$ satisfying the quaternionic relations $IJ=-JI=K$. We choose the metric on $F$ making $(I,J,K)$ orthonormal, and we denote by  $S$ the orthonormal frame bundle of $F$, usually called the Konishi bundle of $N$ \cite{konishi}.

Let $(f_1,f_2,f_3)$ denote the standard basis of $\RM^3$ and let $e_i:=f_j\wedge f_k$ the corresponding basis of $\so(3)$, where here and in the sequel, $\{i,j,k\}$ denotes an arbitrary even permutation of $\{1,2,3\}$. We compute:
\be\label{ef}e_i(f_i)=0,\qquad e_i(f_j)=f_k,\qquad e_i(f_k)=-f_j,\qquad [e_i,e_j]=e_k\ .
\ee

Since $\SO(3)$ is the automorphism group of $\mathbb{H}$, it easily follows that for every $u\in S$, the skew-symmetric endomorphisms $J_i:=uf_i$ are also almost Hermitian structures satisfying the quaternionic relations. 

It was first noticed by Ishihara \cite{ish} that the connection $\beta$ on the principal $\SO(3)$-bundle $\pi:S\to N$ induced by the Levi-Civita covariant derivative of $(N,{g^N})$, has parallel curvature when viewed as section of $\Lambda^2 \T N\otimes \ad(S)$. To make things precise, we choose around each point of $N$ a local section $u=(J_1,J_2,J_3)$ as before, and denote by  $\omega_i:={g^N}(J_i\cdot,\cdot)$ the corresponding locally defined $2$-forms on $N$. The manifold $(N,g^N)$ is Einstein with scalar curvature $\mathrm{Scal}^{N}=4m\kappa$ and Ricci tensor $\Ric^{N}=\kappa {g^N}$ for some non-zero real number $\kappa$, and the curvature tensor of the connection $\beta$ satisfies:
\be\label{9R}R^\beta_{X,Y}J_i=\frac{\kappa}{m+2}(\omega_j(X,Y)J_k-\omega_k(X,Y)J_j)\ ,
\ee
 (see \cite[Eq. (2.13)]{ish} or \cite[Lemma 14.40]{besse}; in the second reference however one should note that the factor 2 appearing there is wrong, and that the  convention for the curvature differs from ours by a sign). 
 
Consider now the local sections $E_i:=u(e_i)$ of the adjoint bundle $\ad(S)=\End^-(F)$. By \eqref{ef}, we have $E_i(J_i)=0$, $E_i(J_j)=J_k$, $E_i(J_k)=-J_j$, and $[E_i,E_j]=E_k$. Using this frame, \eqref{9R} becomes 
\be\label{9R1}R^\beta_{X,Y}=-\frac{\kappa}{m+2}\sum_{\ell=1}^3\omega_\ell(X,Y)E_\ell.
\ee

The scalar product $B$ making the basis $(e_1,e_2,e_3)$ of $\so(3)$ orthonormal is $\ad_{\so(3)}$-invariant, and every $\ad_{\so(3)}$-invariant scalar product on $\so(3)$ is of the form $\la\cdot,\cdot\ra_r:=rB(\cdot,\cdot)$, with $r>0$.
We fix such an $r$ and define the family of metrics $g_r$ on $S$ by 
\be\label{gr}g_r=\pi^*{g^N}+\la\gamma,\gamma\ra_r\ .\ee

\begin{epr}\label{pad} The Riemannian manifold $(S,g_r)$ carries a $3$-$(\alpha,\delta)$-Sasakian structure with $\alpha:=\frac{\kappa\sqrt r}{2(m+2)}$ and $\delta:=\frac1{2\sqrt r}$.
\end{epr}
\begin{proof}
The connection and curvature forms $\beta\in\Omega^1(S,\so(3))$ and $\Omega^\beta\in\Omega^2(S,\so(3))$ can be written with respect to the above basis of $\so(3)$ as $\beta=\sum_{\ell=1}^3\beta_\ell e_\ell$ and $\Omega^\beta=\sum_{\ell=1}^3\Omega^\beta_\ell e_\ell$.
By \eqref{met}, the fundamental vertical vector fields $\xi_\ell:=\frac1{\sqrt r}e_\ell^*$ define an orthonormal basis of the vertical distribution of $S$, and their metric duals $\eta_\ell:=g_r(\xi_\ell,\cdot)$ are given by $\eta_\ell={\sqrt r}\beta_\ell$. 
By \eqref{sts} and \eqref{9R1}, we get $\Omega^\beta_\ell=-\frac{\kappa}{m+2}\pi^*\omega_\ell$. We introduce $\Phi^H_\ell:=-\pi^*\omega_\ell$, $\Phi_i:=\Phi^H_i-\eta_j\wedge\eta_k$ and let $\f_\ell$ denote the endomorphisms of $\T S$ defined by $g_r(U,\f_\ell(V))=\Phi^H_\ell(U,V)$. By \cite[Definition 1.2.2]{ad}, $(g_r,\f_\ell,\xi_\ell,\eta_\ell)_{\ell\in\{1,2,3\}}$ defines an almost 3-contact metric structure on $S$.

From the first structure equation \eqref{str1} we get 
$$\Omega^\beta_i=\d\beta_i+\beta_j\wedge\beta_k\ ,$$
whence
$$\d\eta_i={\sqrt r}\d\beta_i={\sqrt r}(\Omega^\beta_i-\beta_j\wedge\beta_k)={\sqrt r}(\frac{\kappa}{m+2}\Phi^H_i-\frac1r\eta_j\wedge\eta_k)\ .$$

From \cite[Definition 2.2.1]{ad}, $(g_r,\f_\ell,\xi_\ell,\eta_\ell)_{\ell\in\{1,2,3\}}$ is a 3-$(\alpha,\delta)$-Sasakian structure on $S$, with $\alpha:=\frac{\kappa\sqrt r}{2(m+2)}$ and $\delta:=\frac1{2\sqrt r}$. By changing the signs of the vector fields $\xi_\ell$ and of their dual $1$-forms, we also obtain that  $(g_r,\f_\ell,-\xi_\ell,-\eta_\ell)_{\ell\in\{1,2,3\}}$ is a 3-$(\alpha,\delta)$-Sasakian structure, with $\alpha:=-\frac{\kappa\sqrt r}{2(m+2)}$ and $\delta:=-\frac1{2\sqrt r}$. 
\end{proof}

We claim that if $N$ has positive scalar curvature, there exists a unique $r>0$ such that $(g^N,P:=S,\gg:=\so(3),\gamma:=\beta,\langle\cdot,\cdot\rangle_r,\psi)$ is a parallel $\so(3)$-structure on $N$ in the sense of Definition \ref{las1}. 
We need to check that $\psi:\ad(P)\to\Lambda^2\T N$ defined by \eqref{epsi} is a Lie algebra morphism for exactly one value of $r$. 

Since by definition the metric on $\ad(S)$ is induced by the scalar product $\la\cdot,\cdot\ra_r$ on $\so(3)$, we have $\la E_i,E_j\ra=r\delta_{ij}$. Taking  $\sigma:=E_i$ in \eqref{epsi} and using \eqref{9R1}, we get for every tangent vectors $X,Y\in\T N$:
$$g^N(\psi(E_i),X\wedge Y)=-\langle E_i,R^\beta_{X,Y}\rangle=\frac{r\kappa}{m+2}\omega_i(X,Y)=\frac{r\kappa}{m+2}g^N(\omega_i,X\wedge Y)\ ,$$
so $\psi(E_i)=\frac{r\kappa}{m+2}J_i$ as endomorphism of $T N$. Since $[J_i,J_j]=2J_k$, $\psi$ is a Lie algebra morphism if and only if $\frac{r\kappa}{m+2}=\frac12$, i.e. $r=\frac{m+2}{2\kappa}$.

One thus obtains a geometry with parallel skew-symmetric torsion of special type on $S$ given by Proposition \ref{p75}, i.e. by applying Theorem \ref{redu1} in the special case where $\sigma=0$, $P:=S$, $\gg:=\so(3)$, $\gamma:=\beta$, $\k=0$, $\v:=\gg$, $\langle\cdot,\cdot\rangle_\v=rB(\cdot,\cdot)$ for $r=\frac{m+2}{2\kappa}$. Moreover, for this particular value of $r$ one has $\delta=2\alpha=\frac{\sqrt \kappa}{\sqrt{2(m+2)}}$.

This fact, together with Proposition \ref{pad} corroborate the results in \cite{ad}, where the authors construct a connection with parallel skew-symmetric torsion on every 3-$(\alpha,\delta)$-Sasakian manifold, and check that the Sasakian vector fields $\xi_i$ are parallel with respect to this connection if and only if $\delta=2\alpha$. In other words, 3-$(\alpha,\delta)$-Sasakian structures define geometries with parallel skew-symmetric torsion of special type for $\delta=2\alpha$, and this explains that they can be obtained from parallel $\gg$-structures as we just showed. 

For all other values of $\alpha$ and $\delta$ (in particular for $3$-Sasakian manifolds, corresponding to $\alpha=\delta=1$), the Agricola-Dileo connection with parallel skew-symmetric torsion on 3-$(\alpha,\delta)$-Sasakian manifolds is no longer of special type, so it has to fit into the more general setting of geometries with parallel curvature. Let us now explain this in detail.

Consider again a quaternion-Kähler manifold $(N^{4m},{g^N})$ as before, together with the principal $\SO(3)$-bundle $\pi:S\to N$ endowed with the connection $\beta$ induced by the Levi-Civita connection of $(N,g)$. The product $P:=S\times \SO(3)$ has the structure of a principal $\SO(3)\times \SO(3)$-bundle over $N$. 

Let us denote by $\beta_0\in\Omega^1(P,\so(3))$ the pull-back to $P$ through the second projection $P\to\SO(3)$ of the Maurer-Cartan form of $\SO(3)$. Then $\gamma:=(\beta,\beta_0)\in\Omega^1(P,\so(3)\oplus\so(3))$ is a connection form on $P$, with curvature form $\Omega^\gamma=(\Omega^\beta,0)$. We denote by $\gg:=\so(3)\oplus\so(3)$, $\k:=\{(\zeta,\zeta)\ |\ \zeta\in\so(3)\}$, and for $a\in\RM\setminus\{1\}$, $\v_a:=\{(\zeta,a\zeta)\ |\ \zeta\in\so(3)\}$. Then for every $r>0$, the decomposition $\gg=\k\oplus\v_a$ is naturally reductive with respect to the scalar product $\la\cdot,\cdot\ra_r$ on $\v_a$ defined by $\la(\zeta_1,a\zeta_1),(\zeta_2,a\zeta_2)\ra_r:=rB(\zeta_1,\zeta_2)$.

We claim that $(N,g,\sigma:=0,P,\gg,\gamma,\k,\v:=\v_a,\langle\cdot,\cdot\rangle_\v:=\langle\cdot,\cdot\rangle_r)$ is a geometry with parallel curvature (Definition \ref{pgwt}) if and only if $r$ and $a$ are related by $2r\kappa=(1-a)(m+2)$. We only need to check when does condition $(iii)$ in Definition \ref{pgwt}, i.e. \eqref{s2}, hold.

For $\zeta\in\so(3)$, the $\v$-component of $(\zeta,0)$ with respect to the decomposition $\gg=\k\oplus\v$ is 
\be\label{comp}(\zeta,0)_\v=\frac{1}{1-a}(\zeta,a\zeta).
\ee 
Moreover, if $\pmb{u}=(u,h)\in P$ for some $u\in S$ and $h\in \SO(3)$, then for every $X,Y\in\T N$, we have $\pmb{u}^{-1}R^\gamma_{X,Y}=(u^{-1}R^\beta_{X,Y},0)$. Consequently, for every $\pmb{\zeta}:=(\zeta,a\zeta)\in\v$, the endomorphism $R^\gamma_{\pmb{u}\pmb{\zeta}}$ defined in \eqref{og} satisfies:
\bea g^N(R^\gamma_{\pmb{u}\pmb{\zeta}} (X), Y)&=&\la (\pmb{u}^{-1}R^\gamma_{X,Y})_\v, \pmb{\zeta} \ra_\v =\la (u^{-1}R^\beta_{X,Y},0)_\v, (\zeta,a\zeta)\ra_r\\
&=&\frac{1}{1-a}\la (u^{-1}R^\beta_{X,Y},au^{-1}R^\beta_{X,Y}), (\zeta,a\zeta) \ra_r=\frac{r}{1-a}B(u^{-1}R^\beta_{X,Y},\zeta).
\eea
Denoting $uf_i:=J_i$ and $ue_i=:E_i$ as before, \eqref{9R1} yields 
\be\label{urb}u^{-1}R^\beta_{X,Y}=-\frac{\kappa}{m+2}\sum_{\ell=1}^3\omega_\ell(X,Y)e_\ell,
\ee
whence
$$g^N(R^\gamma_{\pmb{u}\pmb{\zeta}} (X), Y)=\frac{r}{1-a}B(u^{-1}R^\beta_{X,Y},\zeta)=-\frac{r\kappa}{(1-a)(m+2)}\sum_{\ell=1}^3B(\zeta,e_\ell)\omega_\ell(X,Y)\ ,$$
thus showing that 
\be\label{rz}R^\gamma_{\pmb{u}\pmb{\zeta}}=\frac{r\kappa}{(1-a)(m+2)}\sum_{\ell=1}^3B(\zeta,e_\ell)J_\ell.\ee

By bilinearity and skew-symmetry, \eqref{s2} holds if and only if
\be\label{39bis}g^N ([R^\gamma_{\pmb{u}\pmb{\zeta}_2},R^\gamma_{\pmb{u}\pmb{\zeta}_1}](X),Y)+\langle[u^{-1}R^\gamma_{X,Y},\pmb{\zeta}_2]_\v,\pmb{\zeta}_1\rangle_r=0\ ,\ee
for  $\pmb{\zeta}_1:=(e_i,ae_i)$ and  $\pmb{\zeta}_2:=(e_j,ae_j)$, such that $\{i,j,k\}$ is an even permutation of $\{1,2,3\}$. 
Using \eqref{rz} we compute:
\bea g^N ([R^\gamma_{\pmb{u}\pmb{\zeta}_2},R^\gamma_{\pmb{u}\pmb{\zeta}_1}](X),Y)&=&\left(\frac{r\kappa}{(1-a)(m+2)}\right)^2g^N ([J_j,J_i](X),Y)\\&=&-2\left(\frac{r\kappa}{(1-a)(m+2)}\right)^2\omega_k(X,Y),
\eea
and by \eqref{comp} and \eqref{urb}:
\bea\langle[u^{-1}R^\gamma_{X,Y},\pmb{\zeta}_2]_\v,\pmb{\zeta}_1\rangle_r&=&-\frac{\kappa}{m+2}\sum_{\ell=1}^3\omega_\ell(X,Y)\langle[(e_\ell,0),(e_j,ae_j)]_\v,\pmb{\zeta}_1\rangle_r\\&=&-\frac{\kappa}{m+2}\sum_{\ell=1}^3\omega_\ell(X,Y)\langle([e_\ell,e_j],0)_\v,(e_i,ae_i)\rangle_r\\
&=&-\frac{\kappa}{(1-a)(m+2)}\sum_{\ell=1}^3\omega_\ell(X,Y)\langle([e_\ell,e_j],a[e_\ell,e_j]),(e_i,ae_i)\rangle_r\\
&=&\frac{r\kappa}{(1-a)(m+2)}\omega_k(X,Y).
\eea
Comparing the last two relations shows that \eqref{39bis} holds if and only if 
\be\label{rka}2r\kappa=(1-a)(m+2).\ee 

By Theorem \ref{redu1}, whenever this equality holds, we obtain a geometry $(g^M,\tau)$ with parallel skew-symmetric torsion on the quotient $M:=P/K$.

We claim that $(M,g^M)$ also carries a 3-$(\alpha,\delta)$-Sasakian structure. Indeed, denoting by $\pi_M:P\to M$ the canonical projection, and by $\iota:S\to P$ the map $u\mapsto (u,1)$, the composition $\f:=\pi_M\circ\iota$ is a diffeomorphism from $S$ to $M$. We now compute the pull-back by $\f$ of the metric $g^M$ on $M$ defined in Theorem \ref{redu1}. By \eqref{eh} we get for every $X,Y\in\T_u S$:

\bea (\f^*g^M)(X,Y)&=&\iota^*((\pi_M)^*g^M)(X,Y)=((\pi_N)^*g^N)((\iota_*X),(\iota_*Y))+\la\gamma(\iota_*X)_\v,\gamma(\iota_*Y)_\v\ra_r\\
&=&g^N(\pi_*X,\pi_*Y)+\la(\beta(X),0)_\v,(\beta(Y),0)_\v\ra_r\\
&=&g^N(\pi_*X,\pi_*Y)+\frac{1}{(1-a)^2}\la(\beta(X),a\beta(X)),(\beta(Y),a\beta(Y))\ra_r\\
&=&g^N(\pi_*X,\pi_*Y)+\frac{r}{(1-a)^2}B(\beta(X),\beta(Y))\ .
\eea

Consequently, $\f^*g^M$ is the metric $g_s$ on $S$ defined in \eqref{gr} for $s:=\frac{r}{(1-a)^2}$. By Proposition \ref{pad}, $(M,g^M)$ carries a 
3-$(\alpha,\delta)$-Sasakian structure on $S$, with $\alpha:=\frac{\kappa\sqrt r}{2(m+2)|1-a|}$ and $\delta:=\frac1{2\sqrt r|1-a|}$. 

When the parameters $r,\kappa$ and $a$ satisfy the constraint \eqref{rka} (i.e. for $1-a=\frac{m+2}{2r\kappa}$), we thus have shown that $(M,g^M)$ has a geometry with parallel skew-symmetric torsion, and a 3-$(\alpha,\delta)$-Sasakian structure, with $\alpha:=\frac{\kappa}{4\sqrt r|\kappa|}$ and $\delta:=\frac{m+2}{4r\sqrt r|\kappa|}$. When $\kappa$ runs through $\RM^*$ and $r$ through $\RM^*_+$, we obtain all possible pairs $(\alpha,\delta)$ in $\RM^*\times \RM^*_+$. Like before, a metric with 3-$(\alpha,\delta)$-Sasakian structure also has a 3-$(-\alpha,-\delta)$-Sasakian structure, so the above construction covers all pairs $(\alpha,\delta)$ in $\RM^*\times \RM^*$. Note that the sign of the scalar curvature of the quaternion-Kähler basis $N$ is always equal to the sign of $\alpha\delta$, see also \cite{ad}.

\bigskip

\labelsep .5cm


\begin{thebibliography}{22}
{\footnotesize

\bibitem{agricola}
I. Agricola, A.C. Ferreira, Th. Friedrich,
{\sl The classification of naturally reductive homogeneous spaces in dimensions $n\le6$},
Differential Geom. Appl. {\bf 39} (2015), 59--92.

\bibitem{ad}
I. Agricola, G. Dileo,
{\sl Generalizations of $3$-Sasakian manifolds and skew torsion},
Adv. Geom. {\bf 20} (3) (2020), 331--374.

\bibitem{alex1}
B. Alexandrov, Th. Friedrich, N. Schoemann, 
{\sl Almost Hermitian $6$-manifolds revisited}, 
J. Geom. Phys. {\bf 53} (2005), 1--30.

\bibitem{alex2}
B. Alexandrov, 
{\sl $\Sp(n)\U(1)$-connections with parallel totally skew-symmetric torsion}, 
J. Geom. Phys. {\bf 57} (2006), 323--337.

\bibitem{ambrose}
W. Ambrose, I.M. Singer, 
{\sl On homogeneous Riemannian manifolds}, 
Duke Math. J.  {\bf 25} (1958), 647--669.

\bibitem{andrei}
F. Belgun, A. Moroianu, 
{\sl Nearly-K\"ahler $6$-manifolds with reduced holonomy}, 
Ann. Global Anal. Geom. {\bf 19} (2001), 307--319.

\bibitem{besse}
A. Besse, {\it Einstein manifolds},
 Ergebnisse der Mathematik und ihrer Grenzgebiete (3)  {\bf 10}
 Springer-Verlag, Berlin, 1987. 
 
 \bibitem{but} J.-B. Butruille,  {\sl Classification des
	vari{\'e}t{\'e}s approximativement k{\"a}hleriennes ho\-mo\-g{\`e}nes},
Ann. Global Anal. Geom. {\bf 27} (2005), 201--225.

\bibitem{c-swann} 
R. Cleyton, A. Swann, 
{\sl Einstein metrics via intrinsic or parallel torsion}, Math. Z. {\bf 247} (2004), 513--528.

\bibitem{fh} L. Foscolo, M. Haskins, {\sl New $\G_2$-holonomy cones and exotic nearly Kähler structures on the $6$-sphere and the product of two $3$-spheres}, Ann. Math. {\bf 185} (1) (2017), 59--130.

\bibitem{friedrich1}
Th. Friedrich, 
{\sl $\G_2$-manifolds with parallel characteristic torsion}, 
Differ. Geom. Appl. {\bf 25} (2007), 632--648.

\bibitem{friedrich2}
Th. Friedrich, S. Ivanov, 
{\sl Parallel spinors and connections with skew-symmetric torsion in string theory}, 
Asian J. Math. {\bf 6} (2002), 303--335.

\bibitem{fiau}
Th. Friedrich, I. Kath, A. Moroianu, U. Semmelmann,
{\sl On nearly parallel $\G_2$-structures},
J. Geom. Phys. {\bf 23} (3-4) (1997), 259--286.

\bibitem{gray1}
A. Gray, 
{\sl Riemannian manifolds with geodesic symmetries of order $3$}, 
J. Differential Geometry {\bf 7} (1972), 343--369.

\bibitem{ish}
S. Ishihara,
{\sl Quaternion K\"ahlerian manifolds},
J. Differential Geometry {\bf 9} (1974), 483--500.

\bibitem{ivanov}
S. Ivanov, 
{\sl Connections with torsion, parallel spinors and geometry of $\Spin(7)$ manifolds},
Math. Res. Lett. {\bf 11} (2-3) (2004), 171--186. 

\bibitem{kiri}
V. Kirichenko, 
{\sl K-spaces of maximal rank}, 
Matem. Zametki  {\bf 22} (4) (1977), 465--476.

\bibitem{kn1}
S. Kobayashi, K. Nomizu, 
{\sl Foundations of differential geometry},
Vol I. Interscience Publishers, John Wiley \& Sons, New York-London 1963 xi+329 pp. 

\bibitem{kn2}
S. Kobayashi, K. Nomizu, 
{\sl Foundations of differential geometry}, 
Vol. II. Interscience Publishers John Wiley \& Sons,  New York-London 1969 xv+470 pp.

\bibitem{konishi}
M. Konishi, 
{\sl On manifolds with Sasakian $3$-structure over quaternion Kähler manifolds}, 
Kōdai Math. Sem. Rep. {\bf 26} (1974/75), 194--200.

\bibitem{merkulov}
S. Merkulov, L. Schwachh\"ofer,
{\sl Classification of irreducible holonomies of torsion-free affine connections}, 
Ann. of Math. (2) {\bf 150} (1) (1999), 77--149.

\bibitem{cliff}
A. Moroianu, U. Semmelmann, 
{\sl Clifford structures on Riemannian manifolds}, 
Adv. Math.  {\bf 228} (2) (2011), 940--967.

\bibitem{nagy}
 P.-A. Nagy, 
 {\sl Nearly K\"ahler geometry and Riemannian foliations},
 Asian J. Math. {\bf 6} (3) (2002), 481--504.
 
\bibitem{nom} K. Nomizu, {\sl Invariant affine connections on homogeneous spaces}, Amer. J. Math. {\bf 76} (1954), 33--65.

\bibitem{schoeman}
N. Schoemann, 
{\sl Almost Hermitian structures with parallel torsion}, 
J. Geom. Phys. {\bf 57} (2007), 2187--2212.

\bibitem{storm2}
R. Storm,
{\sl A new construction of naturally reductive spaces}, 
Transform. Groups {\bf 23} (2) (2018), 527--553.

\bibitem{storm1}
R. Storm,
{\sl
The classification of naturally reductive homogeneous spaces in dimension $7$ and $8$},
PhD thesis, Universit\"at Marburg (2017).

}
\end{thebibliography}
\end{document}